\documentclass[11pt,reqno]{amsart}

\usepackage{amsmath,amssymb,amsfonts}
\usepackage[left=30mm,right=30mm,top=30mm,bottom=30mm]{geometry}
\usepackage{hyperref}
\usepackage{cite}
\usepackage[foot]{amsaddr}
\usepackage[sc]{mathpazo}
\usepackage{graphics}
\usepackage{graphicx}
\usepackage{color}
\usepackage{empheq} 
\usepackage{cases}
\usepackage{comment}

\newcommand{\Z}{\mathbb{Z}}

\newcommand{\R}{\mathbb{R}}

\renewcommand{\L}{\mathsf{L}^2}
\renewcommand{\H}{\mathsf{H}}
\newcommand{\HS}{\mathcal{H}}

\newcommand\deps{\Lambda\e}

\renewcommand{\a}{\mathfrak{a}}
\newcommand{\A}{{\mathcal{A}}}
\newcommand{\Res}{{\mathcal{R}}}
\newcommand{\J}{{J}}

\newcommand{\dom}{\mathrm{dom}}
\newcommand{\supp}{\mathrm{supp}}

\newcommand{\eps}{\varepsilon}
\newcommand{\e}{_{\varepsilon}}
\newcommand{\ie}{_{i,\eps}}  

\newcommand{\al}{\alpha}

\newcommand{\ga}{\gamma}

\renewcommand{\d}{\,\mathrm{d}}
\newcommand{\ds}{\displaystyle}

\newcommand{\I}{\mathcal{I}}
\newcommand{\Id}{\mathrm{I}}

\newcommand{\cupl}{\bigcup\limits}
\newcommand{\suml}{\sum\limits}
\newcommand{\liml}{\lim\limits}

\newcommand{\wt}{\widetilde}
  
\newcommand{\ceq}{\coloneqq}

\theoremstyle{plain}
\newtheorem{theorem}{Theorem}[section]
\newtheorem*{theorem*}{Theorem}
\newtheorem{lemma}[theorem]{Lemma}
\newtheorem*{lemma*}{Lemma}

\theoremstyle{remark}
\newtheorem{remark}[theorem]{Remark}
\newtheorem*{remark*}{Remark} 
\newtheorem{example}[theorem]{Example}
\newtheorem*{example*}{Example} 
\theoremstyle{definition}

\numberwithin{equation}{section}

\begin{document}
\title
[Operator estimates for homogenization of the Robin Laplacian  in a perforated domain]
{Operator estimates for homogenization of the Robin Laplacian in a perforated domain}

\author[Andrii Khrabustovskyi]{Andrii Khrabustovskyi$^{1,2}$}
\address{$^1$ Department of Physics, Faculty of Science, University of
  Hradec Kr\'{a}lov\'{e}, Czech Republic} 
\address{$^2$ Department of Theoretical Physics,
Nuclear Physics Institute of the Czech Academy of Sciences, \v{R}e\v{z}, Czech Republic} 
\email{andrii.khrabustovskyi@uhk.cz}

\author[Michael Plum]{Michael Plum$^3$}
\address{$^3$ Institute of Analysis, Faculty of Mathematics, Karlsruhe Institute of Technology, Germany} 
\email{michael.plum@kit.edu}

\thispagestyle{empty}


\begin{abstract}
Let $\varepsilon>0$ be a small parameter. We consider the domain $\Omega_\varepsilon:=\Omega\setminus D_\varepsilon$, where $\Omega$ is an open domain in $\mathbb{R}^n$,   and $D_\varepsilon$ is a family of small balls  of the radius $d_\varepsilon=o(\varepsilon)$ distributed periodically with period $\varepsilon$. Let $\Delta_\varepsilon$ be the Laplace operator in  $\Omega_\varepsilon$ subject to the  Robin condition ${\partial u\over \partial n}+\gamma_\varepsilon u = 0$ with $\gamma_\varepsilon\ge 0$  on the boundary of the holes  and the Dirichlet condition on the exterior boundary. Kaizu (1985, 1989) and Brillard (1988) have shown that, under appropriate assumptions on $d_\varepsilon$ and $\gamma_\varepsilon$, the operator $\Delta_\varepsilon$ converges in the strong resolvent sense to the sum of the Dirichlet Laplacian in $\Omega$ and a constant potential. We improve this result deriving  estimates on the rate of convergence in terms of  $L^2\to L^2$ and $L^2\to H^1$ operator norms. As a byproduct we establish the estimate on  the distance between the spectra of the associated operators.  
\end{abstract}


\maketitle

\section{Introduction}\label{sec:1}

We revisit one of the classical problems in homogenization theory -- homogenization of the Laplacian in a domain with a lot of tiny holes. For the Dirichlet Laplacian it is also  known as \emph{crushed ice problem}. Here we focus on holes with the Robin boundary conditions. 
In the introduction we  recall the setting of the problem along with some important known results, and then sketch the main outcomes of the present work.

\subsection{Homogenization of Robin Laplacian in perforated domains}\label{subsec:1:1}
Let $\eps>0$ be a small parameter.
We consider the following perforated domain:
$$\Omega\e\ceq \Omega\setminus \Big(\cupl_i \overline{D\ie}\Big).$$
Here $\Omega$ is a fixed open domain in $\R^n$ ($n\geq 2$), $D\ie\Subset\Omega$ are identical open balls of the radius $d\e$
distributed evenly in
$\Omega$ along an $\eps$-periodic lattice.  The domain $\Omega\e$ is given in Figure~\ref{fig1}. A more accurate 
description of $\Omega\e$ is postponed to the next section. In the following, we refer to the sets $D\ie$ as \emph{holes}.

In the present paper we focus on the case of a low concentration
of the holes, namely
\begin{gather}\label{Lambda}
\deps\ceq {d\e\over \eps}\to 0\text{ as }\eps\to 0.
\end{gather}

We consider the following boundary-value problem in $\Omega\e$:
\begin{gather}\label{BVP:e}
\begin{cases}
-\Delta u\e + u\e =f&\quad\text{in }\Omega\e,\\
\ds{\partial u\e\over\partial n} +\gamma\e u\e=0&\quad\text{on }\cup_i\partial  D\ie,\\
u\e=0&\quad\text{on }\partial\Omega.
\end{cases}
\end{gather}
Here $f\in\L(\Omega)$ is a given function, ${\partial\over\partial n}$ stands for the derivative along the
exterior unit normal  to the boundary $\partial\Omega\e$, and the  constant $\gamma\e$ satisfies $0\leq \gamma\e<\infty$. 
Homogenization theory is aimed to describe the behavior of the solution $u\e$ to this problem as $\eps\to 0$.

\begin{figure}[h]
\center{
\begin{picture}(230,160)
\includegraphics[width=0.5\textwidth]{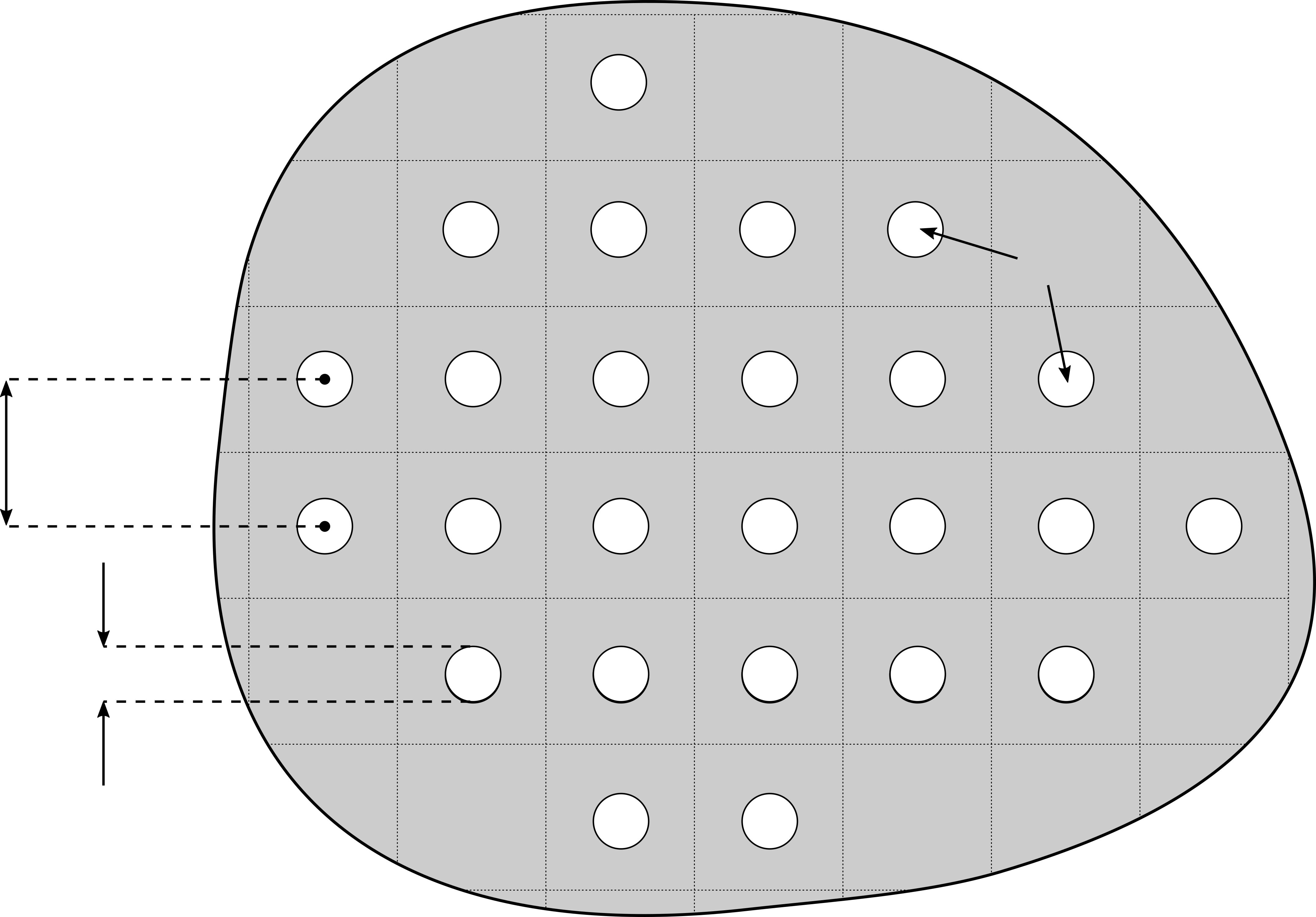}

\put(-220,37){$2d\e$}
\put(-227,75){$\eps$}
\put(-48,110){$D\ie$}

\end{picture}}
\caption{The domain $\Omega\e$}\label{fig1}
\end{figure} 
 
In was demonstrated in \cite{Ka89,Br88} that the asymptotic behavior of $u\e$ is determined by the asymptotic  
behaviour of two quantities $P\e\ge 0$ and $Q\e>0$ given by
\begin{gather}\label{peqe}
\ds
{P\e}\ceq  \varkappa_n{\ga\e d\e^{n-1}\over \eps^{n}},\qquad
{Q\e}\ceq 
\begin{cases}\ds
  (n-2)\varkappa_n{d\e^{n-2}\over \eps^{n}},&n\ge 3,\\[2mm]\ds
 {2\pi \over |\ln d\e| \eps^2},&n=2,
\end{cases}
\end{gather}
with $\varkappa_n$ standing for the surface area of the unit sphere in $\R^n$.
We define
\begin{gather}
\label{pq}
{P}\ceq \lim_{\eps\to 0}P\e,\quad  {Q}\ceq \lim_{\eps\to 0} Q\e,\quad P,Q\in [0,\infty]
\end{gather}
(here we assume that both limits in \eqref{pq}, either finite or not, exist).
Two cases are possible \cite{Ka89,Br88}:
if $P=\infty$ and $Q=\infty$, then 
\begin{gather}
\label{scenario:1}
\left\|u\e\right\|_{\L(\Omega\e)}\to 0\quad\text{as}\quad\eps\to 0,
\end{gather}
otherwise, if  $P<\infty$ or $Q<\infty$, then
\begin{gather}
\label{scenario:2}
\left\|u\e-u\right\|_{\L(\Omega\e)}\to 0\quad\text{as}\quad\eps\to 0,
\end{gather}
where $u$ is the solution to the problem
\begin{gather}\label{BVP:0}
\begin{cases}
-\Delta u  + Vu + u =f&\text{in }\Omega,\\
u=0&\text{on }\partial\Omega,
\end{cases}
\end{gather}
with the constant potential $V$ defined by 
\begin{gather*}
V\ceq 
 \left\{
\begin{array}{llcl}
P,&P>0&\text{and}&Q=\infty,\\[1mm]
Q,&P=\infty&\text{and}&Q>0,\\[1mm]
PQ(P+Q)^{-1},&P>0&\text{and}&Q>0,\\[1mm]
0,&P=0&\text{or}&Q=0.
\end{array}
\right.
\end{gather*}
It is easy to see that
\begin{gather}
\label{Ve}
V=\liml_{\eps\to 0}V\e,\quad\text{where }
V\e\ceq {P\e Q\e\over P\e+Q\e }. 
\end{gather}

\begin{example}\label{example:1}
Let $n\geq 3$. Let 
$ d\e=\eps^s$ with $s>1$ (this restriction is caused by the assumption \eqref{Lambda}), $\gamma\e=\eps^t$ with $t\in\R$. 
Then
$P\e=\varkappa_n\eps^{t+(n-1)s-n}$ and $Q\e=(n-2)\varkappa_n\eps^{(n-2)s-n}$.
In Figure~\ref{fig2} we sketch five subsets
of the set of admissible parameters $\{(s,t)\in\R^2:\ s>1\}$. If $(s,t)$ belongs to the dark gray area (respectively, the light gray area), then \eqref{scenario:1} holds (respectively, \eqref{scenario:2}--\eqref{BVP:0} with $V=0$ hold). If $(s,t)$ is on the dashed bold open interval (respectively, on the solid bold open ray), one has \eqref{scenario:2}--\eqref{BVP:0} with $V=\varkappa_n$ (respectively, with $V=(n-2)\varkappa_n$). Finally, we have \eqref{scenario:2}--\eqref{BVP:0} with $V=\varkappa_n{n-2\over n-1}$ provided
$s=-t={n\over n-2}$.

\begin{figure}[h]
\center{
\begin{picture}(300,150)
\includegraphics[width=0.42\textwidth]{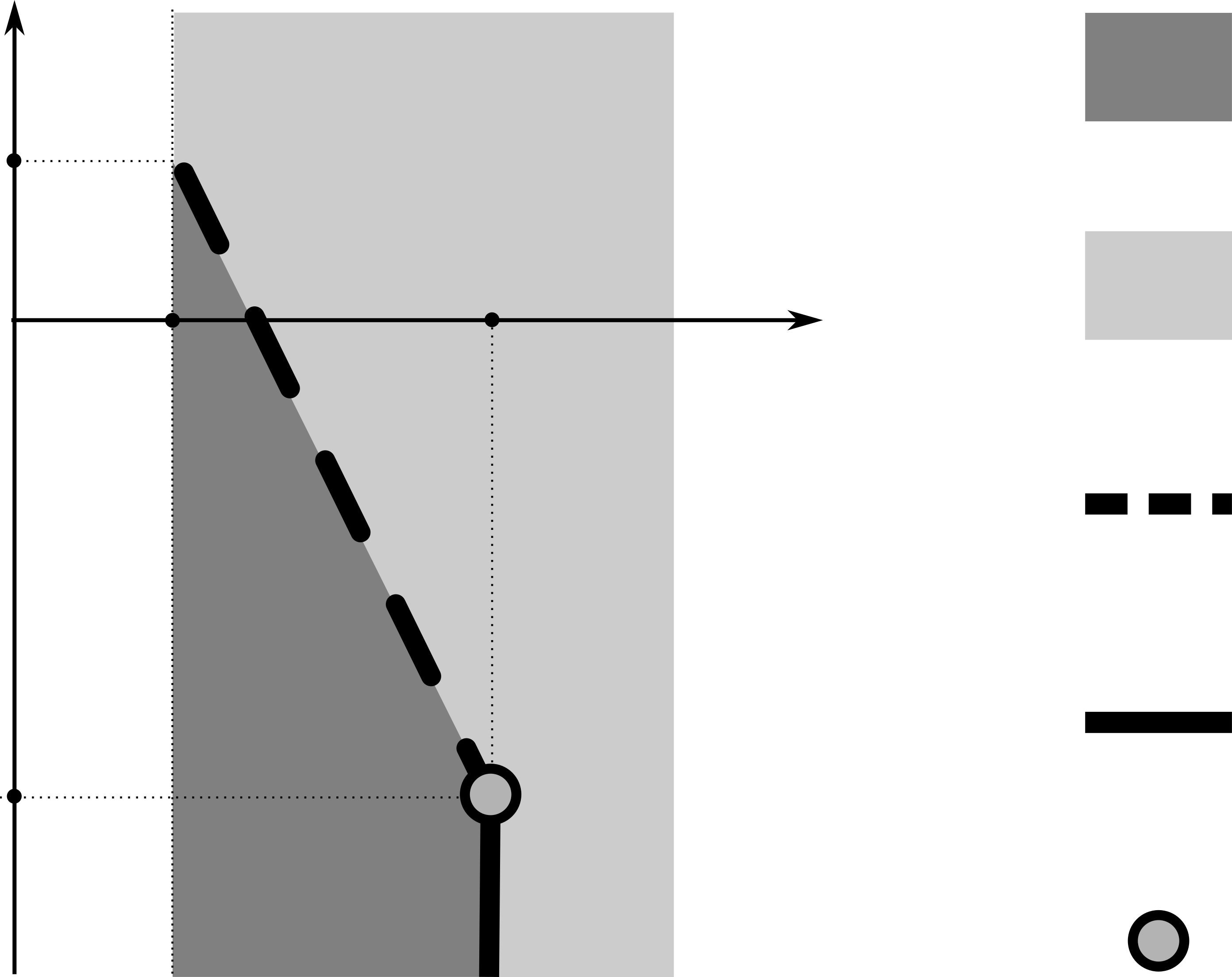}
\put(10,135){$P=\infty$ and $Q=\infty$}
\put(10,102){$P=0$ or $Q=0$} 
\put(10,70){$P>0$ and $Q=\infty$}
\put(10,36){$P=\infty$ and $Q>0$}
\put(10,3){$P>0$ and $Q>0$}  

\put(-165,87){$^1$}
\put(-130,87){${n\over n-2}$}
\put(-159,87){${n\over n-1}$}
\put(-203,23){${-n\over n-2}$}
\put(-190,118){$^1$}

\put(-190,141){$^t$}
\put(-75,89){$^s$}

\end{picture}}
\caption{Example~\ref{example:1}\label{fig2}}
\end{figure}

\end{example}

For the case
$(\ga\e=\gamma\,\wedge\, P>0)$ 
the above result was obtained by Kaizu in \cite{Ka85a}; in Figure~\ref{fig2} it corresponds to the point $ ({n\over n-1},0)$. The case $(Q>0\,\wedge\,P>0)$ was studied in \cite{Ka85b,Ka86} for $n=3$.
All other cases were investigated by Brillard in \cite{Br88} and, using another methods, by  Kaizu  in \cite{Ka89}. Actually, the articles \cite{Ka86,Br88,Ka89} dealt with the  holes of 
the form $D\ie\cong d\e D$, where the  set $D$ may have an arbitrary shape (i.e., it is not necessary a ball).
As for the  ball-shaped holes, 
$u\e$ converges to zero if $P=Q=\infty$;
otherwise, if $P<\infty$ or $Q<\infty$, $u\e$ converges to the solution $u$ of the boundary value 
problem \eqref{BVP:0} with an appropriately modified constant potential $V$.
Namely,
if $P>0$ and $Q=\infty$, then $V=|\partial D|\varkappa_n^{-1}P$, where $|\partial D|$ is the surface area of $D$; 
if $P=\infty$ and $Q>0$, then $V=\mathrm{cap}(D)\varkappa_n^{-1}(n-2)^{-1}Q$ if $n\ge3$, where $\mathrm{cap}(D)$ is the Newton capacity of $D$ (see, e.g., \cite{T96} for its definition),  and $V=Q$ if $n=2$; 
if $P>0$ and $Q>0$, then $V$ is the minimum of some capacity-type functional depending on $P,Q$ and $D$ (cf.~\cite[Proposition~3.3]{Br88}); if either $P=0$ or $Q=0$, then  $V=0$ as before. 
In the current work, in order to make the presentation simpler, we 
restrict ourselves to the ball-shaped holes.

The same homogenization problem with non-periodically distributed holes or/and with a non-constant $\gamma\e$ was studied, e.g., in \cite{Ka90,BG90,DGSZ19}. Qualitatively, the main result in these papers remains the same as in  \cite{Ka89,Br88}, but
the potential $V$ in \eqref{BVP:0} may be non-constant.  

Earlier, Marchenko and Khruslov~\cite{MK64}, and
Cioranescu and Murat~\cite{CM82}
considered a similar problem for  $\gamma\e=\infty$ which  
corresponds to 
Dirichlet conditions $u\e=0$ on $\partial D\ie$.
In these articles quite general shapes and distributions of holes are allowed.
Notably, the result is similar to the one for Robin conditions, and for identical ball-shaped holes distributed $\eps$-periodically
it reads as follows: if $Q=\infty$, then \eqref{scenario:1} holds, 
otherwise one has \eqref{scenario:2}--\eqref{BVP:0} with $V=Q$.
The cases $Q=0$ and $Q=\infty$ were also considered in \cite{RT75}, moreover,
the authors investigated a similar problem for randomly distributed holes under assumptions resembling  
$Q>0$.

Despite in the present paper we deal with holes which are asymptotically smaller than the period (see~\eqref{Lambda}),
in this introduction we would like to devote some attention  to the case when they are of the same order. 
Namely, let
$d\e=r\eps$ with $r\in (0,\frac12)$.  
We define $P,Q$ as in \eqref{pq}, \eqref{peqe}; 
clearly, $Q=\infty$. Like in the case \eqref{Lambda},
$u\e$ converges to zero (cf.~\eqref{scenario:1}) provided $P=\infty$.
Otherwise, if $P<\infty$, $u\e$ converges in the sense \eqref{scenario:2} to  $u\in \H^1_0(\Omega)$ 
satisfying
\begin{gather*}
-\mathrm{div}(A \nabla u)  + Pu + Bu = Bf.
\end{gather*}
Here   $B>0$ is the measure of the set $\overline{\square\setminus B_{r}}$,
where $\square$ is the unit cube, and $B_r$ is the ball of  radius $r$, both having the same center;
$A\gg 0$ is a constant matrix,
whose entries are calculated by solving a certain boundary value problem on $\square\setminus \overline{B_r}$. 
For $\gamma\e=0$ (Neumann holes) this result was obtained in \cite{CS79},
for $\gamma\e>0$ (Robin holes) -- in 
\cite{Br88,Ka89}. Similar results for more general
shapes and distributions of holes can be found, e.g., in \cite{Kh79,BG90}.

It worth to mention that in \cite{Ka89} (see also the earlier articles \cite{Ka85b,Ka86})  
not only linear, but also 
\emph{non-linear} Robin boundary conditions ${\partial u\e\over\partial n} + \gamma\e g(x,u\e)=0$ were treated. 
In this case a nonlinear term $V(u)$   appears in the limiting equation. 
After \cite{Ka89},  a huge number of articles concerning non-linear Robin problems in perforated domains appeared. 
Not pretending to present an exhaustive overview of these articles here, we refer only to some of them -- \cite{Go95,DGSZ17,GK16,GPS12,MS09}, more results and references can be found in the recent monograph \cite{DGS21}.

Finally, 
 one can also study a \emph{surface distribution of holes}, i.e., holes
  being located near some hypersurface $\Gamma$ intersecting $\Omega$.
  For the first time this problem was considered in~\cite{MK64} for  Dirichlet holes.  
  Robin holes case was treated  in \cite{LOPS97}; the complete analysis including non-linear case can be found in \cite{GPS12}.  
  The result reads as follows: 
the solution $u\e$ to the problem \eqref{BVP:e} converges in $\L(\Omega\e)$ to the function $u$   
  (i.e., \eqref{scenario:2} holds), where $u\in\H^1_0(\Omega)$ satisfies 
  $-\Delta u + u = f$ in $\Omega\setminus\Gamma$ and certain interface conditions on $\Gamma$. 
 
\subsection{Main results}\label{subsec:1:3}

{
In all mentioned above papers 
concerning homogenization of the linear boundary value problems in domains with a lot of tiny holes
(cf.~
\cite{Br88,Ka89,Ka90,BG90,DGSZ19,MK64,CM82,LOPS97,GPS12,Ka85a,Kh79})
the convergence result  was established for the fixed right-hand-side $f\in\L(\Omega)$.
In the language of operator theory this means that one has  \emph{strong resolvent convergence}  
 of the differential operators associated with the boundary value problems \eqref{BVP:e} and \eqref{BVP:0} \footnote{Strictly
speaking, we are not able to treat the classical resolvent convergence,
since the operators act in different Hilbert spaces $\L(\Omega\e)$ and $\L(\Omega)$. 
Nevertheless, using the operator $\J\e f\ceq  f\restriction_{\Omega\e}$,
one gets its natural analogue for varying domains $\Omega\e
\subset \Omega$, see Section~\ref{sec:2}.}. 
In the current work we improve \eqref{scenario:1}--\eqref{scenario:2}, by proving the 
uniform convergence with respect to
$f\in \{g\in\L(\Omega):\  \|g\|_{\L(\Omega)}=1\}$. In another words, we establish  \emph{norm resolvent convergence} of the underlying operators. Moreover, we estimate the rate of this 
convergence (the so-called \emph{operator estimates}). }

Recall that $u\e$ and $u$ stand for the solutions to the problems \eqref{BVP:e} and \eqref{BVP:0}, respectively.
Our first result reads as follows. Let $P<\infty$ or $Q<\infty$. Then one has the estimate
\begin{gather}\label{NRC:1}
\|u\e-u\|_{\L(\Omega\e)}\leq C\eta\e \|f\|_{\L(\Omega)},\quad
\eta\e\to 0\text{ as }\eps\to 0.
\end{gather}
Here $C$ is a positive constant independent of $\eps$, and 
$\eta\e$ is given below in \eqref{eta}. 
As a byproduct of \eqref{NRC:1}, we obtain the estimate for the distance between the spectra of the underlying operators.

Besides the convergence in the $\L(\Omega)\to \L(\Omega\e)$ operator norm,
in some cases (see~\eqref{goodPQ})
we also prove the convergence in the $\L(\Omega)\to \H^1(\Omega\e)$ operator norm. Namely,
we derive the estimate
\begin{gather}\label{NRC:2}
\|u\e-u\|_{\H^1(\Omega\e)}\leq 
C\eta\e'\|f\|_{\L(\Omega)},\text{ where }\eta\e'\ceq \max\left\{V\e Q\e^{-1/2};\,\eta\e\right\}
\end{gather} 
(recall that $Q\e$ and $V\e$ are given in \eqref{peqe} and \eqref{Ve}, respectively).
Evidently, one has  
\begin{gather}
\label{Vpq}
V\e Q\e^{-1/2}\leq \min\{Q\e^{1/2};\, P\e Q\e^{-1/2}\},
\end{gather} 
whence the right-hand-side of \eqref{NRC:2} converges to zero provided either
\begin{gather}\label{goodPQ}
(P<\infty\ \wedge\ Q=\infty)\quad\text{or}\quad(P=0\ \wedge\ Q>0)\quad\text{or}\quad Q=0.
\end{gather}
To get a reasonable estimate in the remaining case $(P\not= 0\, \wedge \, Q>0)$ we need a special corrector;
the corresponding result reads as follows:
\begin{gather*}
\|u\e-(1+G\e)u \|_{\H^1(\Omega\e)}\leq C\wt\eta\e\|f\|_{\L(\Omega)},\
\quad\wt\eta\e\to 0.
\end{gather*}
Here $G\e$ is a smooth function given in \eqref{Ge},
while $\wt\eta\e$ is defined in \eqref{wtdeltan}.

In the last section we also discuss possible improvements of 
  \eqref{NRC:1}--\eqref{NRC:2} for $P,Q$ satisfying
\begin{equation}\label{spec.case}
(P<\infty\ \wedge\ Q=\infty)\quad\text{or}\quad(P=0\ \wedge\ Q>0).
\end{equation}

Finally, we derive the estimate
\begin{gather}\label{NRC:4}
\|u\e\|_{\L(\Omega\e)}\leq 
C\max\left\{P\e^{-1};\,Q\e^{-1};\,\eps^2\right\}\|f\|_{\L(\Omega)},
\end{gather}
giving the required convergence result if $P=Q=\infty$.

The main results are collected in Theorems~\ref{th1}, \ref{th2}, \ref{th3}, \ref{th4}, \ref{th5} and \ref{th2a}. 
We formulate them in operator terms.
Our proofs (except the one of \eqref{NRC:4}) rely on the abstract results from \cite{KP21,AP21,P06} 
concerning the resolvent and spectral convergence   in varying Hilbert spaces.

\subsection{Previous works on operator estimates in homogenization}\label{subsec:1:2}

Operator estimates in homogenization theory is a rather young
  topic. It was initiated by  Birman and  Suslina \cite{BS04,BS07},  Griso \cite{G04,G06}, Zhikov and  Pastukhova \cite{Z05a,Z05b,ZP05} (see also the overview \cite{ZP16}) 
  for   
  the classical homogenization problem concerning elliptic
  operators of the form
  $\mathrm{div}\left(A (\frac . \eps)\nabla\right)$,
   where $A(\cdot)$ is a $\mathbb{Z}^n$-periodic function.

For Dirichlet holes ($\gamma\e=\infty$) an operator estimate was established 
by the first author and Post in \cite{KP18} under the assumption $Q<\infty$. In a subsequent work
 Ann\'e and Post \cite{AP21} derived operator estimates 
for the Laplace-Beltrami operator on a Riemannian manifold with Dirichlet holes 
satisfying the conditions resembling either the case  $Q=0$ or the case $Q=\infty$,
and for Neumann holes satisfying (a kind of) condition \eqref{Lambda}.
For the Neumann Laplacian in periodically perforated domains with holes of the same
smallness order as the period operator estimates were obtained in \cite{Z05b,Su18}. 
We also mention the work \cite{CDR18}, where
the norm resolvent convergence without estimates on its rate was demonstrated for three particular cases: 
Neumann holes ($\gamma\e=0$) satisfying \eqref{Lambda},
Dirichlet holes satisfying $Q\e=Q>0$, and
Robin holes with $\gamma\e=\gamma$, $d\e=\eps^{n/(n-1)}$.

For the surface distribution of holes (see the end of the previous subsection) some operator estimates
were derived in  \cite{BCD16,GPS13}. 
In \cite{GPS13} operator estimates were obtained 
for the Robin Laplacian in a bounded domain $\Omega\subset\mathbb{R}^3$ with a lot of small
 holes located in a neighborhood of a plane intersecting $\Omega$; we note that
 the restrictions on $\Omega$ (boundedness) and on the dimension ($n=3$) 
played an important role in the proofs in \cite{GPS13}.
In   \cite{BCD16}
operator estimates were derived for
elliptic operators posed in a two-dimensional domain with a lot of small
 holes located in a neighborhood of a  curve; various boundary conditions were treated, in particular, the Robin conditions with $\eps$-independent coupling constant $\gamma$.

Finally, we mention the closely related papers  
\cite{BBC10} and \cite{BCFP13}, where operator estimates were deduced, respectively,
  for elliptic operators with frequently alternating boundary
  conditions and for boundary value problems in domains with
  oscillating boundary.

Let us stress that, in contrast to \cite{Br88,Ka89},  in this paper we do not assume that  $\Omega$ is bounded. We also have no restrictions on the dimension $n$ as in \cite{BCD16,GPS13}.

\section{Setting of the problem and main results}\label{sec:2}

Let $\Omega $ be an open domain in $\R^n$,  $n\geq 2$.
We assume that $\partial\Omega$ is $\mathsf{C}^2$ smooth, moreover, if $\Omega$ is unbounded,
we additionally require it to be
\emph{uniformly regular of class $\mathsf{C}^2$} in the sense of Browder \cite[Definition~1]{Br61}. 
The latter  is automatically fulfilled, for example, for domains with compact smooth boundaries or for compact smooth perturbations of half-spaces.
This requirement is merely technical, and is needed solely to enjoy the global $\H^2$-regularity of solutions to the homogenized 
problem.  If $\Omega$ is bounded, our results remain valid under less restrictive
assumptions on $\partial\Omega$, see Remark~\ref{remark:smooth} at the end of Section~\ref{sec:5}.

Let $\eps$ be a positive parameter.
We assume that $\eps$ is small enough, namely $\eps\in (0,\eps_0]$,
where $\eps_0\in (0,1)$ will be specified later; see \eqref{de}--\eqref{lnln}. 
We denote  
$$
\begin{array}{l}
\square\ceq (-1/2,1/2)^n,\qquad
\square\ie\ceq \eps(\square+i),\ i\in\Z^n,\qquad
\I\e\ceq\left\{i \in \Z^n:\ \square\ie \subset \Omega\right\},
\end{array}
$$
i.e., $\I\e$ is the set of those indices $i\in\Z^n$ for which the corresponding cell
$\square\ie$ lies entirely in $\Omega$.
For $i\in\I\e$ we define $x\ie\ceq i\eps$ and
\begin{gather*}
D\ie\ceq \left\{x\in\R^n:\ |x-x\ie|<d\e\right\}, 
\end{gather*}
where
$d\e\in (0,{\eps/2})$.
Finally,  we set  
\begin{gather*}
  \Omega\e \ceq \Omega \setminus\left( \bigcup_{i \in \I\e}\overline{ D\ie}\right).
\end{gather*} 
The domain $\Omega\e$ is depicted on Figure~\ref{fig1}.

Next we describe the family of operators $\mathcal{A}\e$ which  will be the main object of our interest. Let $\gamma\e\ge 0$.
We   introduce the sesquilinear form $\a\e$ in the Hilbert space $\L(\Omega\e)$  by
\begin{equation}
\label{ae}
\a\e[u,v]=
\ds\int_{\Omega\e}\nabla u\cdot\overline{\nabla v} \d x+
\sum_{i\in\I\e}\gamma\e\int_{\partial D\ie}u \overline{v}\d s,\quad 
\dom(\a\e)=\left\{u\in \H^1(\Omega\e):\, u\restriction_{\partial\Omega}=0\right\}.
\end{equation} 
where $\d s$ is the surface measure on ${\partial D\ie}$.
The second term in \eqref{ae} is indeed finite for $u\in\H^1(\Omega\e)$ which follows easily from  the trace inequality 
\begin{gather*}
\|u\|_{\L(\partial D\ie)}\leq  
C\e \|u\|_{\H^1(\square\ie\setminus\overline{D\ie})},\ u\in\H^1(\Omega\e),
\end{gather*}
where the constant $C\e>0$ is independent of $i\in\I\e$.
Furthermore, it is also straightforward to check that the form 
$\a\e[u,v]$ is symmetric, densely defined, closed, and positive. Then, by the first representation theorem \cite[Chapter 6, Theorem 2.1]{Ka66}, there exists the unique self-adjoint and positive operator $\A\e$ associated with $\a\e$,
i.e., $\dom(\A\e)\subset\dom(\a\e)$ and
\begin{gather*}
(\A\e u,v)_{\L(\Omega\e)}= \a\e[u,v],\quad\forall u\in
\dom(\A\e),\ \forall  v\in \dom(\a\e).
\end{gather*}
The operator $\A\e$ is the Laplace operator in $\Omega\e$ subject to the Dirichlet conditions
on $\partial\Omega$ and the Robin conditions ${\partial u\over \partial n}+\gamma_\varepsilon u = 0$ on 
the boundary of the holes $D\ie$, $i\in\I\e$. The solution to the boundary value problem \eqref{BVP:e} is given by 
$u\e=(\A\e+\Id)^{-1} f\e$, where $f\e\ceq f\restriction_{\Omega\e}$.

Recall that 
$P\ceq \liml_{\eps\to 0}P\e$, $Q\ceq \liml_{\eps\to 0}Q\e$, $V\ceq \liml_{\eps\to 0}V\e,$
where $P\e$, $Q\e$ are defined in \eqref{peqe}, and $V\e$ is given in \eqref{Ve}.
\smallskip

\emph{We start from the case when either $P<\infty$ or $Q<\infty$.}
At first, we introduce the limiting operator $\A$.
Let $\a$ be a sesquilinear form in $\L(\Omega)$  given by
\begin{gather}\label{a}
{\a}[u,v]=\int_{\Omega} \nabla u\cdot \overline{\nabla v}\d x +
V\int_{\Omega}  u\overline{v}\d x,\quad
\dom(\a)=\H^1_0(\Omega).
\end{gather}
We denote by $\A$ the self-adjont  operator associated with this form.
Evidently, $$\A=-\Delta_{\Omega}  +V,$$ 
where $\Delta_\Omega$ is the Dirichlet Laplacian on $\Omega$.
Standard theory of elliptic PDEs (see, e.g., \cite{Ev98}) yields 
$\dom(\A)\subset\H^2_{\rm loc}(\Omega)$. Furthermore, due to the uniform regularity of $\Omega$ \cite{Br61}, 
one has a global $\H^2$-regularity of functions from $\dom(\A)$, namely,
$\dom(\A)=\H^2(\Omega)\cap\H^1_0(\Omega)$.

By  $\J\e$ we denote the operator of restriction to $\Omega\e$, i.e.,  
\begin{gather}\label{J}
\J\e:\L(\Omega)\to \L(\Omega\e),\quad \J\e f\ceq f\restriction_{\Omega\e}.
\end{gather}
Note that $\J\e$ can also be regarded as an operator from $\H^1(\Omega)$ to $\H^1(\Omega\e)$.

Let us specify the range for the small parameter $\eps$. 
In view of \eqref{Lambda}, \eqref{pq} and the equality $|\ln d\e|^{-1}=(2\pi)^{-1} Q\e \eps^{2}$, which holds for $n=2$, there exists  $\eps_0\in (0,1)$ such  that
\begin{gather}
\label{de}
{\sup} _{\eps\in (0,\eps_0]} \Lambda\e\leq 1/4,\\
\label{pq+} 
\begin{cases}
\sup_{\eps\in (0,\eps_0]}P\e<\infty,&\text{provided }P<\infty,\\
\sup_{\eps\in (0,\eps_0]}Q\e<\infty,&\text{provided }Q<\infty,\\
\sup_{\eps\in (0,\eps_0]}P^{-1}\e<\infty,&\text{provided }P=\infty,\\
\sup_{\eps\in (0,\eps_0]}Q^{-1}\e<\infty,&\text{provided }Q=\infty,\\
\inf_{\eps\in (0,\eps_0]}Q\e>0,&\text{provided }Q>0,
\end{cases}
\\
\label{lnln}
{\sup} _{\eps\in (0,\eps_0]} \frac{|\ln\eps|}{|\ln d\e|}\leq 1/2,\text{ provided }n=2\text{ and }Q>0.
\end{gather}
In the following we always assume that $$0<\eps\le\eps_0.$$
Note that $V\e\leq\min\{P\e;\, Q\e\}$,  whence \eqref{pq+} implies
\begin{gather}\label{V+} 
{\sup}_{(0,\eps_0]}V\e<\infty\text{ if  }P<\infty\text{ or }Q<\infty.
\end{gather}

Finally, we define 
\begin{gather}
\label{eta}
\eta\e\ceq
\begin{cases}
\max\left\{|V\e-V|;\,\eps;\,\deps\right\},& n\ge 5,\\[1mm]
\max\left\{|V\e-V|;\,\eps;\,\deps|\ln\deps|\right\},&n=4,\\[1mm]
\max\left\{|V\e-V|;\,\eps;\,\deps^{1/2} \right\},& n= 3,  \\[1mm]
\max\left\{|V\e-V|;\,\eps|\ln\eps|;\,|\ln\deps|^{-{1/2}}\right\},& n=2,
\end{cases}
\end{gather}  
where  $\Lambda\e$ is given in \eqref{Lambda}. 
Due to \eqref{Lambda} and \eqref{Ve}, $\eta\e\to 0$ as $\eps\to 0$.

We are now in position to formulate the   main results of this work.
In the following the notation  $\|\cdot\|_{X\to Y}$ stands  for the  norm of a bounded linear operator acting between Hilbert spaces $X$ and $Y$. By $C$ we denote generic positive constants being independent of $\eps$  (but it may depend on  $n$,  $\Omega$,  $\eps_0$, and on the values of the suprema and the infimum in \eqref{pq+}). 
We denote by $\Id$ the identity operator, either in $\L(\Omega)$ or in $\L(\Omega\e)$.

\begin{theorem}\label{th1}
Let $P<\infty$ or $Q<\infty$.
Then one has  
\begin{gather}\label{th1:est}
\left\|(\A\e+\Id)^{-1}\J\e  - \J\e(\A+\Id)^{-1}\right\|_{\L(\Omega)\to \L(\Omega\e)}\le C \eta\e.
\end{gather} 
\end{theorem}

The next result shows that in some cases 
(see~Remark~\ref{rem:goodPQ} below),
the resolvents of $\A\e$ and $\A$ are also close in the
$(\H^1\to \L)$ operator norm. { We define 
\begin{gather}\label{eta:prime:1}
\eta\e'\ceq \max\left\{V\e Q\e^{-1/2};\,\eta\e\right\},   
\end{gather}
where $\eta\e$ is given in \eqref{eta}, $Q\e$ and $V\e$ are defined in 
\eqref{peqe} and \eqref{Ve}, respectively.}

\begin{theorem}\label{th2}
Let  $P<\infty$ or $Q<\infty$. Then one has 
\begin{gather}\label{th2:est} 
\left\|(\A\e+\Id)^{-1}\J\e  - \J\e(\A+\Id)^{-1}\right\|_{\L(\Omega)\to \H^1(\Omega\e)}\le 
C\eta\e'.
\end{gather}
\end{theorem}

\begin{remark}\label{rem:goodPQ}
For $P,Q$ satisfying \eqref{goodPQ}
the right-hand-side in \eqref{th2:est} tends to zero which follows immediately from the inequality \eqref{Vpq}.
However, in the remaining case 
$
P\not= 0\, \wedge \, Q>0
$
the estimate \eqref{th2:est} is pointless.
To get a reasonable result for this case, 
one needs  a special corrector; see Theorem~\ref{th3} below.
\end{remark}
 
We define the function
${G}\e\in\mathsf{C}^\infty(\Omega\e)$ by
\begin{gather}\label{Ge}
G\e (x)\ceq 
\suml_{i\in\I\e} G\ie(x)\phi\ie(x),\ x\in\Omega\e.
\end{gather}
Here 
$G\ie$ is the fundamental solution to the Laplace equation in $\R^n$ times the constant $V\e\eps^n$:
\begin{gather}\label{Gie}
G\ie(x)\ceq V\e\eps^n\cdot
\begin{cases}
-\ds{1\over \varkappa_n(n-2)|x-x\ie|^{n-2}},&n\ge 3,\\[3mm]
\ds{1\over 2\pi}\ln |x-x\ie|,&n=2,
\end{cases}
\end{gather} 
the cut-off function $\phi\ie$ is defined via
\begin{gather}\label{Phiie}
\phi\ie(x)\ceq \phi\left({4|x-x\ie|\over\eps}\right),
\end{gather}
where    $\phi\in\mathsf{C}^\infty([0,\infty))$ is  such that
\begin{gather}\label{Phi}
0\le \phi(t)\leq 1,\quad 
\phi(t)=1\text{ as }t\le 1,\quad
\phi(t)=0\text{ as }t\ge 2.
\end{gather}
We shall use the same notation $G\e$ for the operator of multiplication by the function $G\e(x)$.
Finally, we define
\begin{gather}\label{wtdeltan}
\wt\eta\e\ceq
\begin{cases}
\max\{|V\e-V|;\,\eps^{2/(n-2)}\},&n\ge 5,\\
\max\{|V\e-V|;\,\eps|\ln \eps| \},&n= 4,\\
\max\{|V\e-V|;\,\eps\},&n= 3,\\
\max\{|V\e-V|;\,\eps|\ln\eps|\},&n= 2.
\end{cases}
\end{gather}

\begin{theorem}\label{th3}
Let  $Q>0$. Then one has  
\begin{gather}\label{th3:est}
\left\|(\A\e+\Id)^{-1}\J\e  - (\Id+G\e)\J\e  (\A+\Id)^{-1}\right\|_{\L(\Omega)\to \H^1(\Omega\e)}\le C
\wt\eta\e.
\end{gather}   
\end{theorem}
 
\begin{remark}
Let $Q>0$. In this case we have, using \eqref{pq+}:
\begin{gather}\label{Qpositive}
\exists C_1,C_2>0:\quad
C_1\eps^{n}\leq
\mathcal{D}\e\leq C_2\eps^{n},
\end{gather}
as $\eps\in (0,\eps_0]$,
where $$\mathcal{D}\e\ceq
\begin{cases}
d\e^{n-2},&n\ge 3,\\
|\ln d\e|^{-1},&n=2.
\end{cases}
$$
It follows easily from \eqref{Qpositive} (taking into account \eqref{lnln} for $n=2$) that
\begin{gather}\label{eta:eta}
\exists C_3,C_4>0:\quad
C_3\eta\e\leq \wt\eta\e\leq C_4\eta\e.
\end{gather}
\end{remark}

Our next goal is to estimate the distance between the spectra of $\A\e$ and $\A$ in a suitable metric.
Recall that
for closed sets $X,Y\subset\R$  the \emph{Hausdorff distance} between  them is given by
\begin{gather}\label{dH}
d_H (X,Y)\ceq\max\left\{\sup_{x\in X} \inf_{y\in Y}|x-y|;\,\sup_{y\in Y} \inf_{x\in X}|y-x|\right\}.
\end{gather}
The notion of convergence provided by this metric   is too restrictive for our purposes.  Indeed, the closeness  of $\sigma(\A\e)$ and $\sigma(\A)$  in  the metric $d_H(\cdot,\cdot)$ would mean  that  these  spectra  look nearly the same uniformly in the whole  of $[0,\infty)$ -- a situation which is not guaranteed by norm  resolvent  convergence.
To overcome this difficulty, it is convenient to introduce the new metric $\widetilde{d_H}(\cdot,\cdot)$ which is given by
\begin{gather}
\label{dHbar}
\widetilde{d_H}(X,Y)\ceq d_H( \overline{(1+X)^{-1}}, \overline{(1+Y)^{-1}}),\ X,Y\subset[0,\infty),
\end{gather}
where  $(1+X)^{-1}\ceq {\{(1+x)^{-1}:\ x\in X\}}$, $(1+Y)^{-1}\ceq {\{(1+y)^{-1}:\ y\in Y\}}$.
With respect to this metric  two  spectra  can  be  close  even  if  they  differ  significantly  at  high energies. 

\begin{remark}
Let $X\e,\,X$ be closed subsets of $[0,\infty)$.
One can show  (see, e.g., \cite[Lemma~A.2]{HN99}) that $\widetilde{d_H}(X\e,X)\to 0$ if and only if   the following two conditions hold simultaneously:
\begin{itemize}
\item[(i)] $\forall x\in X$ there exists a family $(x\e)_{\eps>0}$ with $x\e\in X\e$ such that $x\e\to x$ as $\eps\to 0 $,

\item[(ii)] $\forall x\in\R\setminus X$ there exist $\delta>0$ such that $X\e\cap (x-\delta,x+\delta)=\varnothing$ for  
small enough  $\eps$.

\end{itemize} 
\end{remark}

\begin{theorem}\label{th4}
Let  $P<\infty$ or $Q<\infty$.
Then one has
\begin{gather} \label{th4:est}
\widetilde{d_H}\left(\sigma(\A\e),\sigma(\A)\right)\leq 
C \eta\e,
\end{gather}
where $\eta\e$ is given in \eqref{eta}.
\end{theorem}

Finally, let $P=Q=\infty$. In this case the solution to the boundary value problem \eqref{BVP:0}
converges to zero. The estimate for the rate of this convergence in the operator norm is given in the theorem below.

\begin{theorem}\label{th5}
Let  $P=Q=\infty$. One has
\begin{gather}\label{th5:est}
\big\|(\A\e+\Id)^{-1} \big\|_{\L(\Omega\e)\to \L(\Omega\e)}\le C\max\left\{P\e^{-1};\,Q\e^{-1};\,\eps^2\right\}.
\end{gather}
\end{theorem}

The remaining part of the work is organized as follows.
In Section~\ref{sec:3} we present the abstract results for studying convergence of operators in varying Hilbert spaces. In Sections~\ref{sec:4} we collect several useful estimates which will be widely used further. In Section~\ref{sec:5} we verify the conditions of the above abstract theorems in our concrete setting. 
The proofs of  Theorems~\ref{th1}, \ref{th2}, \ref{th3}, \ref{th4}, \ref{th5} are completed in
Section~\ref{sec:6}.  
{Finally, in Section~\ref{sec:7} we revisit the case \eqref{spec.case}, for which we establish an additional 
$(\H^1\to \L)$ operator estimate (Theorem~\ref{th2a}); in some cases this estimate provides better convergence rate than the estimates \eqref{th1:est} and \eqref{th2:est}.
}

\section{Abstract tools}\label{sec:3}

In \cite{P06} Post presented an abstract toolkit
for studying convergence of 
operators in varying Hilbert spaces; further it was elaborated in the
monograph~\cite{P12} and the papers \cite{AP21,KP21,PS19,MNP13}.
Originally this abstract framework was developed to study convergence of the Laplace-Beltrami operator on  manifolds which shrink to a graph. 
Recently, it also has shown to be effective for homogenization problems in domains with holes, see \cite{KP18,AP21}.
The proofs of our main theorems (except Theorem~\ref{th5}) rely on results from \cite{KP21,AP21,P06} which we recall below.
\smallskip 

Let $\HS\e$ and $\HS$ be Hilbert spaces.
Note that within this section $\HS\e$ is just a {notation} for some Hilbert space which (in general) differs from the space $\HS$, i.e., the sub-script $\eps$ does not mean that this space depends on a small parameter. Of course, later we will apply the results of this section to {the} $\eps$-dependent space $\HS\e\ceq\L(\Omega\e)$.

Let $\A\e$ and $\A$ be non-negative, self-adjoint, unbounded operators in $\HS\e$ and $\HS$, respectively,
and
$\a\e$ and $\a$ be the associated 
sesquilinear forms. We denote
$$\Res\e\ceq(\A\e+\Id)^{-1},\quad \Res\ceq(\A+\Id)^{-1}.$$
Along with $\HS\e$ and $\HS$, we also define the Hilbert spaces $\HS\e^k$ and $\HS^k$, $k=1,2$ via
\begin{gather} \label{scale}
\begin{array}{ll}
\HS^k\e\ceq\dom(\A\e^{k/2}),& 
  \|u\|_{\HS\e^k}\ceq\|{(\A\e+\Id)^{k/2}u}\|_{\HS\e},\\[2mm]
\HS^k\ceq\dom(\A^{k/2}),& 
  \|f\|_{\HS^k}\ceq\|{(\A+\Id)^{k/2}f}\|_{\HS}.\
\end{array}
\end{gather} 
Note that
\begin{gather}\label{scale+}
\begin{array}{ll}
\HS^1\e=\dom(\a\e),&
\|u\|_{\HS^1\e}^2=\a\e[u,u]+\|u\|^2_{\HS\e},\\[2mm]
\HS^1=\dom(\a),&
\|f\|_{\HS^1}^2=\a[f,f]+\|f\|^2_{\HS}.
\end{array}
\end{gather}
Furthermore,
\begin{gather}\label{scale:est}
\begin{array}{lcl}
\HS^2\subset\HS^1\subset\HS&\text{and}&  \|f\|_{\HS}\leq \|f\|_{\HS^1}\leq \|f\|_{\HS^2},\\[2mm]
\HS\e^2\subset\HS\e^1\subset\HS\e&\text{and}& \|u\|_{\HS\e}\leq \|  u\|_{\HS^1\e}\leq \|  u\|_{\HS^2\e}.
\end{array}
\end{gather}

It is well-known (see, e.g., \cite[Theorem~VI.3.6]{Ka66} or \cite[Theorem~VIII.25]{RS72}), that convergence of sesquilinear forms with \emph{common domain} implies norm resolvent convergence of the associated operators.
The theorem below was established by Post in \cite{P06}, and it can be regarded as a generalization of this fact  to the setting of varying spaces. 
 
\begin{theorem}[{\cite[Theorem~A.5]{P06}}]
\label{thA1}
Let $\J\e \colon \HS\to  \HS\e$, ${\wt\J\e }\colon {\HS\e}\to \HS$ be linear operators such that
\begin{align} 
\label{thA1:0}   
|(u,\J\e f)_{\HS\e} - (\wt\J\e u,f)_{\HS}|\leq \delta\e\|f\|_{\HS}\|u\|_{\HS\e},&& \forall f\in\HS,\, u\in\HS\e.
\end{align}
Also let
${ \J\e ^1} \colon {\HS^1} \to {\HS^1\e},\ 
  {\wt\J\e ^{1}} \colon {\HS\e^1}\to {\HS^1}$
be linear operators satisfying the conditions 
\begin{align}
\label{thA1:1}
\|\J\e^1 f-\J\e f\|_{\HS\e}&\leq \delta\e\|f\|_{\HS^1 },&& \forall f\in \HS^1 ,
\\[1mm] 
\label{thA1:2}
 \|\wt\J\e^1 u - \wt\J\e u \|_{\HS}&\leq 
\delta\e\|u\|_{ \HS\e},&&  \forall u\in \HS^1\e, 
\\[1mm]
\label{thA1:3}
 |\a\e[u,\J^1\e f]-\a[\wt\J^{1}\e u,f]  |&\leq 
\delta\e\|f\|_{\HS^2 }\|u\|_{\HS^1\e},&& \forall f\in \HS^2 ,\ u\in \HS^1\e  
\end{align}
with some  $\delta\e\geq 0$.
Then one has
\begin{align}\label{thA1:result}
\|\Res\e\J\e -\J\e \Res\|_{\HS\to\HS\e}\leq 4\delta\e.
\end{align}
\end{theorem} 

\begin{remark}
In  applications the operators $\J\e$ and $\wt\J\e$ usually appear in a natural way -- 
in our case $\J\e$ is defined in~\eqref{J} and $\wt\J\e$ will be defined  in~\eqref{wtJ}. 
The other two operators $\J^1\e$ and $\wt\J^1\e$ should be constructed as ``almost'' restrictions of $\J\e$ and $\wt\J\e$ to $\HS^1$ and $\HS^1\e$, respectively.
\end{remark}

Recently,  Ann\'{e} and Post \cite{AP21} extended the above result to  
a (suitably sandwiched) resolvent difference regarded  as an operator from 
 $\HS$ to $\HS\e^1$.

\begin{theorem}[{\cite[Proposition~2.5]{AP21}}]\label{thA1+}
Let the operators $\J\e$, $\wt\J\e$, $\J\e^1$, $\wt\J\e^1$ satisfy the conditions \eqref{thA1:0}--\eqref{thA1:3} with some $\delta\e\ge 0$.
Then
\begin{gather}\label{thA1+:est}
\|\Res\e\J\e -\J\e^1 \Res\|_{\HS\to\HS^1\e} \leq 6\delta\e.
\end{gather}
\end{theorem} 
 
\begin{remark}\label{rem:AP21}
In fact, Proposition~2.5 in \cite{AP21} is formulated as follows:
\emph{the estimate \eqref{thA1+:est} holds provided the forms 
$\a\e$ and $\a$ are $\delta\e$-{quasi-unitarily equivalent}}. The latter means that
there exist linear operators  
$\J\e \colon \HS\to  \HS\e$, ${\wt\J\e }\colon {\HS\e}\to \HS$,
${ \J\e ^1} \colon {\HS^1} \to {\HS^1\e}$, ${\wt\J\e ^{1}} \colon {\HS\e^1}\to {\HS^1}$
such that  the conditions \eqref{thA1:0}-\eqref{thA1:3} are fulfilled, and moreover, one has
\begin{gather}\label{thA1:4}
\|f\|_{\HS}\leq (1+\delta\e)\|\J\e f\|_{\HS\e} ,
\quad
\|f-\wt\J\e \J\e f\|_{\HS}\leq \delta\e\|f\|_{\HS^1},
\quad
\|u-\J\e \wt\J\e u\|_{\HS\e}\leq \delta\e\|u\|_{\HS\e^1}.
\end{gather}
However, tracing the proof in \cite[Proposition~2.5]{AP21}, one can easily see that 
the estimates
\eqref{thA1:4} are not  utilized for the deduction of \eqref{thA1+:est}.

The  assumptions \eqref{thA1:4} are  required  to deduce other versions of norm resolvent
convergence. Namely, if  conditions \eqref{thA1:0}--\eqref{thA1:3} are fulfilled,
and, moreover, \eqref{thA1:4} are valid, then
\begin{gather*}
\|\Res\e -\J\e \Res \wt \J\e \|_{\HS\e\to\HS\e}\leq C\delta\e,\quad
\|\wt\J\e\Res\e\J\e - \Res  \|_{\HS\to\HS }\leq C\delta\e;
\end{gather*}
see~\cite[Theorem~A.10]{P06}.  
\end{remark}

\begin{remark}
Tracing the proof of \cite[Theorem~A.5]{P06} one observes that
the estimate \eqref{thA1:result}
remains valid if \eqref{thA1:3} is substituted 
by the weaker condition
\begin{gather}\label{weaker}
  \left|\a\e[u,\J^1\e f]-\a[\wt\J^{1}\e u,f] \right|\leq 
  \delta\e\|f\|_{\HS^2 }\|u\|_{\HS^2\e},\quad 
  \forall f\in \HS^2 ,\ u\in \HS^2\e.
\end{gather} 
Nevertheless, in most of the  applications one is able to establish stronger estimate~\eqref{thA1:3}.
Note that in Theorem~\ref{thA1+} condition \eqref{thA1:3} cannot be replaced by \eqref{weaker}.
\end{remark}

The last theorem gives the estimate for the distance between the spectra of $\A\e$ and $\A$
in the metrics $\widetilde{d_H}(\cdot,\cdot)$.
Recall that the distances $d_H(\cdot,\cdot)$ and $\widetilde{d_H}(\cdot,\cdot)$ are given in \eqref{dH} and \eqref{dHbar}, respectively. Note that due to spectral mapping theorem one has
\begin{gather}
\label{SMT}
\widetilde{d_H}(\sigma(\A\e),\sigma(\A))=d_H(\sigma(\Res\e),\sigma(\Res)).
\end{gather}
It is well-known that the norm convergence of bounded self-adjoint operators in a \textit{fixed} Hilbert space implies the Hausdorff convergence of their spectra.
Indeed, let $\Res\e$ and $\Res$ be bounded normal  operators in a Hilbert space $\HS$, then \cite[Lemma~A.1]{HN99}
$$d_{H}(\sigma(\Res\e),\sigma(\Res))\leq \|\Res\e-\Res\|_{\HS\to\HS}.$$
The theorem below is an analogue of this result for the case of operators acting in different Hilbert spaces.
In the present form it was established recently by the first author and Post in \cite{KP21} (see Remark~\ref{rem:tau}), a slightly weaker  version was proven in \cite{CK19}.
 
\begin{theorem}[{\cite[Theorem~3.5]{KP21}}]\label{thA2}
Let  
 $\J\e \colon\HS\to {\HS\e}$, $\wt\J\e \colon\HS\e\to {\HS} $ be  linear bounded operators satisfying 
\begin{align}
\label{thA2:1}
\|\Res\e\J\e - \J\e \Res \|_{\HS\to \HS\e}&\leq \rho\e , 
\\
\label{thA2:2}
\|\wt\J\e \Res\e - \Res\wt\J\e  \|_{\HS\e\to  \HS}&\leq \wt\rho\e ,
\end{align}
and, moreover,
\begin{align}
\label{thA2:3}
\|f\|^2_{ \HS}&\leq \mu\e \|\J\e  f\|^2_{\HS\e}+\nu\e \,  \a[f,f],\quad \forall f\in \dom( \a),\\
\label{thA2:4}
\|u\|^2_{\HS\e}&\leq \wt\mu\e \|\wt\J\e  u\|^2_{\HS}+\wt\nu\e \,  
\a\e[u,u],\quad \forall u\in \dom( \a\e)
\end{align}
for some positive constants $\rho\e ,\,\mu\e ,\,\nu\e ,\, \wt\rho\e ,\,\wt\mu\e ,\,\wt\nu\e  $.
Then   one has
\begin{gather}\label{thA2:est}
\widetilde{d_H}\left(\sigma(\A\e),\,\sigma(\A)\right)\leq 
\max
\left\{
{\nu\e\over 2}+\sqrt{{\nu\e^2\over 4}+\rho\e^2\mu\e};\,
{\wt\nu\e\over 2}+\sqrt{{\wt\nu\e^2\over 4}+\wt\rho\e^2\wt\mu\e}
\right\}.
\end{gather}
\end{theorem} 

\begin{remark}
\label{rem:tau}
In fact, in \cite{KP21} the obtained estimate reads
\begin{gather}\label{rem:tau:est}
\widetilde{d_H}\left(\sigma(\A\e),\,\sigma(\A)\right) 
\leq 
\max\left\{
\rho\e\sqrt{\mu\e\over \tau};\,{\nu\e\over 1-\tau };\,
\wt\rho\e\sqrt{\wt\mu\e\over \wt\tau};\,{\wt\nu\e\over 1-\wt\tau }
\right\},\
\forall\tau,\wt\tau\in (0,1).
\end{gather}
Minimizing
the right-hand-side of \eqref{rem:tau:est} over $\tau,\wt\tau$, 
one easily arrives at the estimate \eqref{thA2:est}.
\end{remark}

\section{Useful estimates}\label{sec:4}

In this section we collect several estimates which will be used in the proofs of the main theorems.
The first result was established in \cite{MK06}.
Here and in what follows the notation
$|D|$ stands for the Lebesgue measure of the domain  $D\subset\mathbb{R}^n$.

\begin{lemma}[{\cite[Lemma~4.9]{MK06}}]
  \label{lemma:MK06}
  Let $D\subset\R^n$ be a parallelepiped, and let
  $D_1,D_2\subset D$ be measurable subsets with $|D_2|\not= 0$.  Then
  \begin{equation*}\forall g\in \H^1(D):\quad
    \|g\|^2_{\L(D_1)}
    \leq \frac {2|D_1|}{|D_2|}\|g\|^2_{\L(D_2)} +
    C \frac {(\mathrm{diam}(D))^{n+1}|D_1|^{1/n}}{|D_2|}\|\nabla g\|^2_{\L(D)},
  \end{equation*}
  where the constant $C>0$ depends only on the dimension $n$.
\end{lemma}

The next lemma is the  Poincar\'e inequality for the cube $\square\ie$.
 
\begin{lemma}
\label{lemma:Poincare}
One has 
\begin{align}
\label{Poincare}
\forall g\in \H^1  (\square\ie)\text{ with }\int_{\square\ie}g(x)\d x=0:&\quad 
\|g  \|^2_{\L(\square\ie)}\leq \pi^{-2}\eps^2\|\nabla g\|^2_{\L(\square\ie)}.
\end{align}
\end{lemma}

\begin{proof}
Let $\lambda\e$  be the second (first non-zero)
eigenvalue of the Neumann Laplacian on $\square\ie$.
It is easy to compute that
$\lambda\e= \left(\frac \pi \eps \right)^2$. Moreover,
by the min-max principle 
\begin{gather*}
  \lambda\e
  =\min\left\{\frac{\|\nabla g\|^2_{\L(\square\ie)}}{\|g\|^2_{\L(\square\ie)}}:\ 
    g\in\H^1(\square\ie)\setminus\{0\},\ \int_{\square\ie}g(x)\d x=0\right\},
\end{gather*}
whence we immediately obtain the 
inequality \eqref{Poincare}.
\end{proof}
 
Recall that $n\ge 2$ stands for the dimension,  $\square$  is a unit cube in $\R^n$. 
By the Sobolev embedding theorem (see, e.g., \cite[Theorem~5.4 and Remark~5.5(6)]{Ad75})
the space $\H^2(\square)$ is embedded continuously into the space
$\mathsf{L}^{p}(\square)$ 
provided $p$ satisfies 
\begin{gather}
\label{p}
1\leq p\leq \frac{2n}{n-4}\text{\; as\; }n\geq 5,\quad
1\leq p<\infty\text{\; as\; }n=4,\quad
1\le p\le \infty\text{\; as\; }n=2,3.
\end{gather} 
Furthermore, the space $\H^2(\square)$ is embedded continuously into
$\mathsf{W}^{1,p}(\square)$ 
provided $p$ satisfies 
\begin{gather}
\label{p:grad}
1\leq p\leq \frac{2n}{n-2}\text{\; as\; }n\geq 3,\quad
1\leq p<\infty\text{\; as\; }n=2.
\end{gather}
In the following, the constants $C_{n,p}$ (for $p$ satisfying \eqref{p})  and $\wt C_{n,p}$ (for $p$ satisfying \eqref{p:grad}) stand for the norms of these embeddings, respectively.
In  the  lemma below we give two Sobolev-type inequalities for the re-scaled cubes $\square\ie\cong\eps\square$.

\begin{lemma}\label{lemma:sobolev}
One has
\begin{gather}\label{sobolev}
\forall g\in \mathsf{H}^2(\square\ie)\text{ with }\int_{\square\ie}g(x)\d x=0:\ 
\|g\|_{\mathsf{L}^p(\square\ie)}\leq C\cdot C_{n,p}\cdot \eps^{n/p+(2-n)/2}\|g\|_{\mathsf{H}^2(\square\ie)}
\end{gather}
provided $p$ satisfies \eqref{p} (for $p=\infty$ one has the convention $1/p=0$). Moreover,
\begin{gather}\label{sobolev:grad}
\forall f\in \mathsf{H}^2(\square\ie):\ 
\|\nabla f\|_{\mathsf{L}^p(\square\ie)}\leq C\cdot\wt C_{n,p}\cdot\eps^{n/p - n/2}\|f\|_{\mathsf{H}^2(\square\ie)}
\end{gather}
provided $p$ satisfies \eqref{p:grad}. The constant $C$ in \eqref{sobolev}--\eqref{sobolev:grad} equals $(1+\pi^{-2})^{1/2}$.
\end{lemma}

\begin{proof}
The estimate \ref{sobolev} is proven in \cite[Lemma~4.3]{KP18}, therefore we present the proof only for the estimate \eqref{sobolev:grad} (in fact, both proofs are based on similar arguments).

The Sobolev embedding theorem yields
\begin{align}\notag
\forall g\in \mathsf{H}^2(\square):\quad 
\|\nabla g\|_{\mathsf{L}^p(\square)}&\leq \|g\|_{\mathsf{W}^{1,p}(\square)}\leq \wt C_{n,p} \|g\|_{\mathsf{H}^2(\square)}
 \\&=\wt C_{n,p} \left( \|g\|^2_{\L(\square)}
    + \|\nabla g\|^2_{\L(\square)}
    + \suml_{k,l=1}^n\left\|{\partial^2g\over \partial x_k\partial x_l}\right\|^2_{\L(\square)}
  \right)^{1/2}\label{sob-square:grad}
\end{align}
provided \eqref{p:grad} holds.
Making the change of variables 
$\square\ni y= x\eps^{-1}-i\text{ with }x\in\square\ie$ we reduce  \eqref{sob-square:grad} to the following estimate:
\begin{align}\notag
\forall g\in\mathsf{H}^2(\square\ie):\quad  \eps^{(p-n)/p}\|\nabla g\|_{\mathsf{L}^p(\square\ie)}
  &\leq \wt C_{n,p} \Bigg(\eps^{-n}\|g\|^2_{\L(\square\ie)}   +\eps^{2-n}\|\nabla g\|^2_{\L(\square\ie)}
  \\
  &\label{sob-square:grad+}  +\eps^{4-n}\suml_{k,l=1}^n\left\|{\partial^2g\over \partial x_k\partial x_l}\right\|^2_{\L(\square\ie)}
  \Bigg)^{1/2}.
\end{align}
Finally, let $f\in \mathsf{H}^2(\square\ie)$.
We denote $\ds f\ie\ceq \eps^{-n}\int_{\square\ie}f(x)\d x$. One has
\begin{gather} \label{fg}
 \nabla (f-f\ie)=\nabla f,\quad{\partial^2 (f-f\ie)\over \partial x_k\partial x_l}={\partial^2 f\over \partial x_k\partial x_l},\quad
 \|f-f\ie\|^2_{\L(\square\ie)}\leq \pi^{-2}\eps^2\|\nabla f\|_{\L(\square\ie)}^2 
\end{gather} 
(the last property follows from Lemma~\ref{lemma:Poincare}). 
 Using \eqref{sob-square:grad+} with $g\ceq f-f\ie$ and taking into account \eqref{fg} and that $\eps<1$, we arrive at the required  estimate 
\eqref{sobolev:grad}.
\end{proof}

In the next section (see the estimate \eqref{pmax+}) we apply the inequality \eqref{sobolev} for the \emph{largest} $p$ satisfying \eqref{p}.
For $n=4$ we are not able to choose a largest $p$ (since in the dimension $4$ the embedding $\H^2\hookrightarrow\mathsf{L}^p$ holds for any $p<\infty$, but not for $p=\infty$), and in this case we need more information on the constant $C_{4,p}$ in the right-hand-side of \eqref{sobolev} -- see the lemma below.

\begin{lemma}\label{lemma:c4p}
For any $p\in [1,\infty)$ 
one has the estimate 
\begin{gather}\label{lemma:c4p:est}
C_{4,p}\leq C  p,
\end{gather}
where the constant $C>0$ is independent of $p$. 
\end{lemma}

\begin{proof}
By the Sobolev embedding theorem, the space    $\mathsf{W}^{1,4}(\R^4)$  is embedded 
continuously into   $\mathsf{L}^p(\R^4)$ provided  $p\in [4,\infty)$; we denote by $C'_p$ the norm of this embedding. Also,  the space $\H^2(\R^4)$ is embedded into $\mathsf{W}^{1,4}(\R^4)$;  the norm of this embedding is denoted by $C''$. Furthermore, by the Calderon extension theorem \cite[Theorem~4.32]{Ad75}, there exists a  linear {bounded} operator 
$E:\H^2(\square)\to \H^2(\R^4)$ such that $(Ef)(x)=f(x)$ a.e. in $\square=(-1/2,1/2)^4$.
Consequently, for any 
 $f\in\mathsf{H}^2(\square)$ 
and $p\ge 4$ one has the following chain of inequalities:
\begin{gather}\label{chain}
\|f\|_{\mathsf{L}^p(\square)}\leq \|Ef\|_{\mathsf{L}^p(\R^4)}\leq 
C'_p\|Ef\|_{\mathsf{W}^{1,4}(\R^4)}\le
C'_p C''\|Ef\|_{\mathsf{H}^2(\R^4)}\leq 
C'_p C''\|E\|\cdot \|f\|_{\mathsf{H}^2(\square)}.
\end{gather}
We will demonstrate below that
\begin{gather}\label{CC}
C'_p\leq  \wt C'_p\ceq   4^{{1\over 4}-{4\nu\over 3p}-{1\over p}}
 \prod_{j=0}^{\nu-1}\left({3p\over 4}-j\right)^{4\over 3p},
\end{gather}
where $\nu$ is the largest integer $\le {3p\over 4}-3$. 
It is easy to see, using Stirling's formula, that
\begin{gather}\label{CC+}
\wt C'_p\leq   \wt C'  p
\end{gather}
with a constant
$\wt C'>0$ being independent of $p$. 
It 
follows   
from \eqref{chain}--\eqref{CC+}
that 
the estimate \eqref{lemma:c4p:est} 
holds with $C\ceq \wt C'C''\|E\|$ provided $p\in [4,\infty)$. Moreover, the H\"older inequality yields
$\|f\|_{\mathsf{L}^{p_1}(\square)}\leq \|f\|_{\mathsf{L}^{p_2}(\square)}$ as
$1\le p_1\leq p_2$, hence \eqref{lemma:c4p:est} is also valid for $p\in [1,4)$.

It remains to prove \eqref{CC}. Let $f\in C_0^\infty(\R^4)$. One has for all $x=(x_1,x_2,x_3,x_4)\in\R^4$:
\begin{align}
\label{Cp:1}
|f(x)|^{3p\over 4}
&\leq {3p\over 4}\int_{-\infty}^{x_1}|f(t,x_2,x_3,x_4)|^{{3p\over 4}-1}\left|{\partial f\over\partial x_1}(t,x_2,x_3,x_4)\right|\d t,
\\
\label{Cp:2}
|f(x)|^{3p\over 4}&\leq {3p\over 4}\int_{x_1}^{\infty}|f(t,x_2,x_3,x_4)|^{{3p\over 4}-1}\left|{\partial f\over\partial x_1}(t,x_2,x_3,x_4)\right|\d t.
\end{align}
Adding the inequalities \eqref{Cp:1} and \eqref{Cp:2} and dividing by $2$, we obtain
\begin{gather}\label{Cp:3}
|f(x)|^{3p\over 4}\leq {3p\over 8}\int_{-\infty}^{\infty}|f(t,x_2,x_3,x_4)|^{{3p\over 4}-1}\left|{\partial f\over\partial x_1}(t,x_2,x_3,x_4)\right|\d t=:F_1(x_2,x_3,x_4).
\end{gather}
In the same way, replacing the variable $x_1$ by any other of the remaining variables, we get
\begin{gather}\label{Cp:4}
|f(x)|^{3p\over 4}\leq F_2(x_1,x_3,x_4),\quad  
|f(x)|^{3p\over 4}\leq F_3(x_1,x_2,x_4),\quad  
|f(x)|^{3p\over 4}\leq F_4(x_1,x_2,x_3),
\end{gather}
with obvious definitions of  $F_2,\,F_3,\,F_4$. 
Multiplying the inequalities in \eqref{Cp:3}-\eqref{Cp:4} and then taking the ${1\over 3}$-power,
one has
\begin{gather}\label{Cp:5}
|f(x)|^{p}\leq \left(F_1(x_2,x_3,x_4)F_2(x_1,x_3,x_4) F_3(x_1,x_2,x_4) F_4(x_1,x_2,x_3)\right)^{1\over 3}.
\end{gather}
Integrating \eqref{Cp:5} with respect to $x_1$ and then applying the   H\"older
inequality,
we obtain
\begin{gather*}
\int_{\R}|f|^{p}\d x_1=
(F_1(x_2,x_3,x_4))^{1\over 3}
\int_{\R}\left(F_2(x_1,x_3,x_4) F_3(x_1,x_2,x_4) F_4(x_1,x_2,x_3)\right)^{1\over 3}\d x_1\\
\leq (F_1(x_2,x_3,x_4))^{1\over 3}
\left(\int_{\R}F_2(x_1,x_3,x_4)\d x_1\right)^{1\over 3}
\left(\int_{\R}F_3(x_1,x_2,x_4)\d x_1\right)^{1\over 3}
\left(\int_{\R}F_4(x_1,x_2,x_3)\d x_1\right)^{1\over 3}.
\end{gather*}
Successively repeating the last steps (i.e., the integration with respect to one of the directions
and the subsequent application of the H\"older inequality with respect to this direction), now for
  the variables $x_2$, $x_3$ and $x_4$, 
we arrive at  
\begin{gather*}
\label{Cp:5+}
\int_{\R^4}|f|^{p}\d x\leq 
\left(
\int_{\R^3}F_1 \d x_2\d x_3\d x_4\cdot 
\int_{\R^3}F_2 \d x_1\d x_3\d x_4\cdot  
\int_{\R^3}F_3 \d x_1\d x_2\d x_4\cdot  
\int_{\R^3}F_4 \d x_1\d x_2\d x_3
\right)^{1\over 3},
\end{gather*}
whence, using the definitions of  the functions $F_k$ and again the H\"older inequality, we derive
\begin{align}\notag
\int_{\R^4}|f |^p\d x
&\leq 
\left({3p\over 8}\right)^{4\over 3}\prod_{k=1}^4
\left(
\int_{\R^4}|f|^{{3p\over 4}-1}\left|{\partial f \over \partial x_k}\right|\d x 
\right)^{1\over 3}\\\notag
&\leq
\left({3p\over 8}\right)^{4\over 3}\prod_{k=1}^4
\left(
\int_{\R^4} |f|^{p-{4\over 3}}\d x 
\right)^{1\over 4}
\left(
\int_{\R^4} \left|{\partial f \over \partial x_k}\right|^4\d x 
\right)^{1\over 12}
\\
&\leq
\left({3p\over 8}\right)^{4\over 3} 
\left(
\int_{\R^4} |f|^{p-{4\over 3}}\d x 
 \right)
\left({1\over 4}\suml_{k=1}^4
\int_{\R^4} \left|{\partial f \over \partial x_k}\right|^4\d x 
\right)^{1\over 3} 
\label{Cp:6}
\le A p^{4\over 3}  \int_{\R^4} |f|^{p-{4\over 3}}\d x,
\end{align}
where
$A\ceq \left({3\over 8}\right)^{4\over 3} \left({1\over 4}\right)^{1\over 3}
\|f\|^{4\over 3}_{\mathsf{W}^{1,4}(\R^4)}$.
Iterating this inequality $\nu-1$ times  we get:
\begin{gather} \label{Cp:7}
\int_{\R^4}|f|^p\d x\le 
A^\nu
\prod_{j=0}^{\nu-1}\left(p-{4j\over 3}\right)^{4\over 3}\int_{\R^4}|f|^{p-{4\nu\over 3}}\d x
\end{gather}
(recall that $\nu$ is the largest integer $\le {3p\over 4}-3$).
Let $q\ceq p-{4\nu\over 3}$. By the choice of $\nu$ we have $4\le q<{16\over 3}$.
Using the H\"older inequality and then applying \eqref{Cp:6} with $p={16\over 3}$, we obtain
\begin{align}\notag
\int_{\R^4}|f|^q\d x=\int_{\R^4}|f|^{16-3q}|f|^{4q-16}\d x
&\leq   \left(\int_{\R^4}|f|^{4}\d x\right)^{4-{3q\over 4}}\left(\int_{\R^4} |f|^{16\over 3}\d x\right)^{{3q\over 4}-3}\\\label{Cp:8}
&\leq A^{{3q\over 4}-3}\left({16\over 3}\right)^{q-4}\int_{\R^4}|f|^4\d x.
\end{align}
Inserting \eqref{Cp:8} into \eqref{Cp:7} and taking into account   
${3q\over 4}={3p\over 4}-\nu$, we get
\begin{align*}  
\int_{\R^4}|f|^p\d x&
\le 
A^{{3p\over 4}-3}\left({16\over 3}\right)^{p-4-{4\nu\over 3}}
\prod_{j=0}^{\nu-1}\left(p-{4j\over 3}\right)^{4\over 3}
 {\int_{\R^4}|f|^4\d x} 
\le (\wt C'_p)^p\|f\|^p_{\mathsf{W}^{1,4}(\R^4)},
\end{align*}
where $\wt C'_p$ is defined in \eqref{CC}. 
By density arguments this estimate holds for all $f\in \mathsf{W}^{1,4}(\R^4)$, which
implies the desired inequality \eqref{CC}. The lemma is proven.
\end{proof}

The following lemma is a simple consequence of the fact that the set
$\Omega\setminus\cupl_{i\in\I\e}\square\ie$ belongs to the $\sqrt{n}\eps$--neighborhood of $\partial\Omega$
(this follows easily from the definition of $\I\e$; note that $\sqrt{n}$ is the length of the diagonal of the cube $\square$).

\begin{lemma}[{\cite[Lemma~4.7]{KP18}}]
\label{lemma:layer}
One has
\begin{gather*}
\forall g\in \H^1_0(\Omega):\quad \|g\|_{\L(\Omega\setminus\cupl_{i\in\I\e}\square\ie)}
\leq C\eps\|\nabla g\|_{\L(\Omega)},
\end{gather*}
where the constant $C>0$ depends only on $\Omega$.
\end{lemma}

{
The  lemma below is used in the proof of Theorem~\ref{th2a}.

\begin{lemma} \label{lemma:tildeg}
One has 
\begin{align}\label{lemma:tildeg:est1}
\forall g\in \H^1(\square\ie):\quad
&\left|g\ie - \wt g\ie\right|^2\leq C\eps^{-n}(\eps^{2}+Q\e^{-1})\|\nabla g\|^2_{\L(\square\ie)}, \\
\label{lemma:tildeg:est2}
\forall g\in \H^1(D\ie):\quad
&\left\|g - \wt g\ie\right\|^2_{\L(\partial D\ie)}\leq 
C d\e\|\nabla g\|^2_{\L(D\ie)}, 
\end{align}
where $g\ie$ and $\wt g\ie$ are the mean values of $g$ over
$\square\ie$ and $\partial D\ie$, respectively, i.e.  
$$
g\ie\ceq \frac1{\eps^{n}}\int_{\square\ie}g\d x,\quad 
\wt g\ie:=\frac{1}{\varkappa_n d\e ^{n-1}}\int_{\partial D\ie}g\d s.
$$
The constant $C>0$ depends only on $n$. 
\end{lemma}

\begin{proof}
The estimate \eqref{lemma:tildeg:est1} is proven in \cite[Lemma~2.1]{Kh13}.
Let us prove \eqref{lemma:tildeg:est2}.
We denote by $D$ the unit ball in $\R^n$. One has the following  trace and Poincar\'e inequalities on $D$:
\begin{align}\label{D:trace:0}
\forall u\in \H^1(D ):&\quad\|u\|^2_{\L(\partial D )}\leq C \| u\|^2_{\H^1(  D )}  ,\\\label{D:Poincare:0}
\forall u\in \H^1(D )\text{ with }\int_{D }u(x)\d x=0:&\quad
\|u\|^2_{\L(  D )}\leq C  \|\nabla u\|^2_{\L(  D )}.
\end{align}
Making  the change of variables $D\ni y= {x -i\eps\over d\e}$ with $x\in D\ie$ we reduce \eqref{D:trace:0}--\eqref{D:Poincare:0}
to
\begin{align}\label{D:trace}
\forall u\in \H^1(D\ie):&\  \|u\|^2_{\L(\partial D\ie)}\leq C \left(d\e^{-1}\|u\|^2_{\L(  D\ie)}  + d\e \|\nabla u\|^2_{\L(  D\ie)}   \right),\\ \label{D:Poincare}
\forall u\in \H^1(D\ie),\ \int_{D\ie}u(x)\d x=0:&\ 
\|u\|^2_{\L(  D\ie)}\leq C  d\e^{2}\|\nabla u\|^2_{\L(  D\ie)}.
\end{align}
One has: 
\begin{multline*}
\|g- \wt g\ie \|_{\L(\partial D\ie)}
 \leq 
\|g- g\ie\|_{\L(\partial D\ie)} + 
|g\ie-\wt g\ie|\cdot |\partial D\ie|^{1/2}\\ =
\|g- g\ie\|_{\L(\partial D\ie)}
 + 
\left|\int_{\partial D\ie} (g\ie-g)\d s\right|\cdot |\partial D\ie|^{-1/2}
\leq 
2\|g- g\ie \|_{\L(\partial D\ie)}, 
\end{multline*}
whence, using \eqref{D:trace}--\eqref{D:Poincare} with $u\ceq g-  g\ie$, 
we deduce  
\begin{gather*}
\|g- \wt g\ie \|_{\L(\partial D\ie)} \leq 
C\left(d\e^{-1}\|g-  g\ie\|^2_{\L(  D\ie)}  + d\e \|\nabla g\|^2_{\L(  D\ie)}   \right)\leq
C  d\e \|\nabla g\|^2_{\L(  D\ie)}.
\end{gather*}
The lemma is proven.
\end{proof}
}

The next estimate plays a crucial role for the proof of Theorems~\ref{th5}.
We denote
\begin{gather}\label{Yie}
Y\ie\ceq \left\{x\in\R^n:\ \ds d\e<|x-x\ie|<{\eps\over 2}\right\}.
\end{gather}

\begin{lemma}
\label{lemma:crucial}
One has
\begin{gather}\label{lemma:crucial:est}
\forall u\in \H^1(Y\ie):\quad 
\|u\|^2_{\L(Y\ie)}\leq 
C\left( P\e^{-1}\gamma\e\|u\|^2_{\L(\partial D\ie)}+
  Q\e^{-1}\|\nabla u\|^2_{\L(Y\ie)}\right),
\end{gather}
where the constant $C>0$ depends only on $n$.
\end{lemma}

\begin{proof}
Evidently, it is enough to prove \eqref{lemma:crucial:est}   only for $C^\infty(\overline{Y\ie})$-functions. We introduce spherical coordinates $(r,\phi)$ in $\overline{Y\ie}$, where
$r\in [d\e,\eps/2]$ is the radial coordinate, which stands for the distance to $x\ie$, while
$\phi=(\phi_1,{\dots},\phi_{n-1})$ are the angular coordinates (here $\phi_j\in [0,\pi]$ as $j=1,\dots,n-2$, 
$\phi_{n-1}\in [0,2\pi)$).
Let $u\in \mathsf{C}^\infty(\overline{Y\ie})$. One has:
\begin{gather}\label{FTC}
u(r,\phi)=u(d\e,\phi)+\int_{d\e}^{r}{\partial
u(\tau,\phi)\over\partial\tau}\d\tau,
\end{gather}
whence
\begin{align*}
|u(r,\phi)|^2&\le 2|u(d\e,\phi)|^2+2\left|\int_{d\e}^{r}{\partial
u(\tau,\phi)\over\partial\tau}\d\tau\right|^2\\ &\leq
2|u(d\e,\phi)|^2+
2I_{\eps} \int_{d\e}^{\eps/2}\left|{\partial
u(\tau,\phi)\over\partial\tau}\right|^2\tau^{n-1}\d\tau,\quad 
\text{where }
I_{\eps}\ceq \int_{d\e}^{\eps/2}\tau^{1-n}\d\tau.
\end{align*}
Multiplying this inequality by $r^{n-1}\d r\d\phi$, where $\d\phi=\left(\prod_{j=1}^{n-2}\left(\sin\phi_j\right)^{n-1-j}\right)\d\phi_1\dots\d\phi_{n-1}$,
and  integrating
over $r\in (d\e,\eps/2)$, $\phi_j\in (0,\pi)$, $j=1,\dots,n-2$, 
$\phi_{n-1}\in (0,2\pi)$,
we get
\begin{align}\notag
\|u\|^2_{\L(Y\ie)}&\leq
2\left(\int_{\mathcal{S}}| u(d\e,\phi)|^2\d\phi  +
 I_{\eps} \int_{\mathcal{S}}\int_{d\e}^{\eps/2}\left|{\partial
u(\tau,\phi)\over\partial\tau}\right|^2\tau^{n-1}\d\tau\d\phi\right)\\
&\times
\left(\int_{d\e}^{\eps/2} r^{n-1}\d r\right)
 \label{crucial:1}
\le
{\eps^n\over 2^{n-1}n } \left( 
d\e^{-n+1}\|u\|^2_{\L(\partial D\ie)}+
 I_{\eps}\|\nabla u \|^2_{\L(Y\ie)}\right)
\end{align}
(here $\mathcal{S}\ceq \left\{\phi=(\phi_1,\dots,\phi_{n-1}):\ 
\phi_j\in (0,\pi),\, j=1,\dots,n-2,\  
\phi_{n-1}\in (0,2\pi)\right\}$).
One has 
\begin{align}
& d\e^{-n+1}= \varkappa_n\gamma\e P^{-1}\e\eps^{-n},
\\
n\ge 3:\quad & \ds
I_{\eps}\leq \int_{d\e}^\infty \tau^{1-n}\d\tau  
=(n-2)^{-1}d\e^{2-n}=
\varkappa_n Q\e^{-1}\eps^{-n}
,\\
n=2:\quad &\ds
I_{\eps}\leq \int_{d\e}^1 \tau^{-1}\d\tau  
=  |\ln d\e|=
2\pi Q\e^{-1}\eps^{-2}.\label{PPQQ}
\end{align}
The desired estimate \eqref{lemma:crucial:est} follows immediately 
from \eqref{crucial:1}--\eqref{PPQQ}.
\end{proof}

{
Swapping $r$ and $d\e$ in  \eqref{FTC} and then repeating 
verbatim the proof of \eqref{lemma:crucial:est} we get the following estimate below; it will be
used in the proof of Theorem~\ref{th2a}.

\begin{lemma}
\label{lemma:crucial+}
One has
\begin{gather}
\label{lemma:crucial:est+}
\forall u\in \H^1(Y\ie):\quad 
\gamma\e\|u\|^2_{\L(\partial D\ie)}
\leq 
C\left( 
P\e\|u\|^2_{\L(Y\ie)}
+
P\e Q\e^{-1}\|\nabla u\|^2_{\L(Y\ie)}\right),
\end{gather}
where the constant $C>0$ depends only on $n$.
\end{lemma}

}

\section{Realisation of the abstract scheme}\label{sec:5}

 {
In this section we apply the abstract Theorems~\ref{thA1}, \ref{thA1+} and \ref{thA2} to the 
homogenization of the problem \eqref{BVP:0}. Through this section we always assume that
\begin{gather}\label{PfinQfin}
P<\infty\text{ or }Q<\infty.
\end{gather}}

We denote $$\HS\e\ceq \L(\Omega\e),\quad \HS\ceq \L(\Omega).$$ 
Recall that the forms $\a\e$ and $\a$ in $\HS\e$ and $\HS$ are defined by \eqref{ae} and \eqref{a}, respectively;
$\A\e$ and $\A$ are non-negative, self-adjoint unbounded operators associated with these forms. 
We introduce the spaces $\HS\e^1$, $\HS^1$ as in \eqref{scale} (cf.~\eqref{scale+}):
\begin{gather}
\begin{array}{ll}
\HS^1\e=\left\{u\in \H^1(\Omega\e):\ u\restriction_{\partial\Omega}=0\right\},
&
\ds\|u\|^2_{\HS^1\e}=\|u\|^2_{\H^1(\Omega\e)}+
\gamma\e\sum_{i\in\I\e}\|u\|_{\L(\partial D\ie)}^2,\\
\HS^1=\H_0^1(\Omega),
&
\ds\|f\|^2_{\HS^1}=\|f\|^2_{\H^1(\Omega)}+
V \|f\|_{\L(\Omega)}^2.
\end{array}\label{H1}
\end{gather}
Finally, we introduce the space $\HS^2$ as in \eqref{scale}:
\begin{gather}\label{H2}
\HS^2=\H^2(\Omega)\cap\H^1_0(\Omega),\quad
\|f\|_{\HS^2}=\|(-\Delta +1)f\|_{\L(\Omega)}.
\end{gather}
Note that, due to uniform regularity of $\Omega$, the norm $\|\cdot\|_{\HS^2}$
is equivalent to the standard norm in $\H^2(\Omega)$; see the estimate \eqref{el:reg} below.
 
Along with   $\J\e:\HS\to\HS\e$ (see \eqref{J}) we   introduce the operator 
$\wt\J\e:\HS\e\to\HS$ via
\begin{gather}\label{wtJ}
(\wt\J\e u)(x)\ceq
\begin{cases}
u(x),&x\in\Omega\e,\\
0,&x\in\Omega\setminus\Omega\e=\overline{\cup_{i\in\I\e}D_{i\e}}.
\end{cases}
\end{gather}

It is known that  there exists a linear operator
$\wt\J\e^1: \H^1(\Omega\e)\to \H^1(\Omega)$
satisfying 
\begin{gather}\label{wtJ1prop}
(\wt\J\e^1 u)\restriction_{\Omega\e}=u,\qquad
\|\wt\J\e^1 u\|_{\H^1(\Omega)}\leq C\| u\|_{\H^1(\Omega\e)},\ \forall u\in \H^1(\Omega\e).
\end{gather}
As before, by $C$ we denote a generic positive constant  independent of $\eps$ and of functions occurring at the estimates  where these constants occur; it may vary from line to line.  
For the construction of such  an operator   we refer, e.g., to \cite[Example~4.10]{MK06}
\footnote{For the construction of the  operator satisfying  \eqref{wtJ1prop},
the holes have to be sufficiently far away from each other; namely,  
$\mathrm{dist}(D\ie,\cup_{i\not=j}D_{j,\eps})\geq \wt C d\e$ should hold.  
In our case (see~\eqref{de})  this inequality holds  with $\wt C=2$.}.
In the following, we shall use the notation  $\wt\J\e^1$ for the restriction of the operator $\wt\J\e^1$ to $\dom(\a\e)=\{u\in\H^1(\Omega\e):\ u\restriction_{\partial\Omega}=0\}$. Evidently, $\mathrm{ran}(\wt\J\e^1)\subset\H^1_0(\Omega)$,
thus $\wt\J\e^1$ is a well-defined linear operator from $\HS\e^1$ to $\HS^1$.
It follows from \eqref{wtJ1prop} that
\begin{gather}\label{wtJ1prop+}
\|\wt\J\e^1 u\|_{\H^1(\Omega)}\leq C\| u\|_{\HS^1\e},\ \forall u\in \dom(\a\e).
\end{gather}

To employ Theorems~\ref{thA1} and \ref{thA1+},
 we also need  a suitable operator $\J\e^1:\HS^1\to\HS\e^1$. For $f\in \H^1_0(\Omega)$ we define $\J\e^1 f$ as follows:
\begin{gather}\label{J1}
(\J\e^1 f)(x) \ceq 
f(x) +\suml_{i\in\I\e}\left[(f\ie-f(x))\wt\phi\ie(x) +  f\ie G\ie(x) \phi\ie(x)\right],\
x\in\Omega\e.
\end{gather}
Here
\begin{itemize}

\item $f\ie$ is the mean value of the function $f(x)$ in the domain $\square\ie$, i.e.,
\begin{gather*}\label{fie}
f\ie\ceq \eps^{-n}\int_{\square\ie}f(x)\d x,
\end{gather*}

\item for $n\geq 3$ the function $\wt\phi\ie $ is given by
\begin{gather*}
\wt\phi\ie\ceq \phi\left({|x-x\ie|\over d\e}\right),
\end{gather*}
where  $\phi:[0,\infty)\to\R$ is a smooth function satisfying \eqref{Phi},

\item for $n=2$ the  function $\wt\phi\ie$ is given by
\begin{equation*}\wt\phi\ie(x)\ceq  \
\begin{cases}
\dfrac{\ln|x-x\ie|-\ln \sqrt{\eps d\e}}{\ln d\e -\ln \sqrt{\eps d\e}}&
   \text{as }\, |x-x\ie|\in (d\e,\sqrt{\eps d\e}),\\
0&\text{as }\, |x-x\ie|\geq \sqrt{\eps d\e},
\end{cases}
\end{equation*}

\item 
$G\ie$, $\phi\ie\in C^\infty(\overline{\Omega\e})$ are given in  \eqref{Gie}, \eqref{Phiie}, respectively.
   
\end{itemize}
Note that 
\begin{gather}\label{supp:phi}
\supp(\wt\phi\ie(x))\subset\overline{Y\ie},\quad
\supp(\phi\ie(x))\subset\overline{Y\ie},
\end{gather}
where $Y\ie$ is given in \eqref{Yie} (for $n=2$ the first inclusion follows from \eqref{de}).
It is easy to see that if $f\in\H^1_0(\Omega)$, then 
$\J\e^1 f\in \H^1(\Omega\e)$ and $\J\e^1 f=0$ on $\partial\Omega$. Thus $\J\e^1$ is indeed a well-defined linear operator from $\HS^1$ to $\HS^1\e$.

It is straightforward to check that 
\begin{gather}\label{cond0:final}
(u,\J\e f)_{\HS\e} - (\wt\J\e u,f)_{\HS}=0,\qquad \forall f\in\HS,\, u\in\HS\e,
\end{gather}
i.e., the condition \eqref{thA1:0}  holds true with any $\delta\e\ge 0$.
Our  goal is to show that the operators $\J\e,\,\wt\J\e,\,\J\e^1,\,\wt\J\e^1$   satisfy the conditions 
\eqref{thA1:1}--\eqref{thA1:3}  with $\delta\e\leq C\eta\e$; here $\eta\e$ is given in \eqref{eta}.
We prove this in Subsections~\ref{subsec:5:1}--\ref{subsec:5:4}.

{
\begin{remark}
The function $\J\e^1 f$ defined by \eqref{J1} resembles  special
test-functions one uses in the so-called \emph{Tartar's energy method} (see, e.g., \cite[Chapter~8]{CD99}). For homogenization problems in  domains with tiny holes this method was used
by Cioranescu and Murat in \cite{CM82}. Our present construction \eqref{J1} is inspired by
test-functions elaborated in \cite{BK97,BCK97} for homogenization problems on manifolds with complicated microstructure.
We adjusted $\J\e^1$ in such a way that for $f\in \H^2(\Omega)$ 
the function $u\ceq \J\e^1 f$ satisfies the Robin boundary conditions $\partial_{n}u+\gamma\e u = 0$
on $\partial D\ie$ (this follows easily from the properties of the cut-off functions $\wt\phi\ie$, and $\phi\ie$ and the
behavior of $G\ie$ on $\partial D\ie$, see~\eqref{G:bk} below). 
\end{remark}

\begin{remark}
One can also define the operator $\J\e^1$ similarly to the operator $\J\e$, i.e. 
\begin{gather}\label{J1:new}
J\e^1 f = f\restriction_{\partial\Omega\e}. 
\end{gather}
In Section~\ref{sec:7} we show that, if $\J\e^1$ is defined by \eqref{J1:new} and $P,Q$ satisfy  \eqref{spec.case},
then the conditions of Theorem~\ref{thA1} are fulfilled  with $\delta\e= C\eta''\e\to 0$, where $\eta''\e $ is defined by  \eqref{eta:prime:2}. In some cases $\eta\e''$ gives better convergence rates comparing with the one in \eqref{th1:est} and \eqref{th2:est}. 
\end{remark}
}

\subsection{Check of condition \eqref{thA1:1}}\label{subsec:5:1}

Let $f\in\dom(\a)=\H^1_0(\Omega)$.
Recall that  the set $Y\ie$ is defined  by \eqref{Yie}.
Using $|\wt\phi\ie|\le 1$, $|\phi\ie|\le 1$ and \eqref{supp:phi}, we get
\begin{align}
\notag
\|\J\e^1 f -\J\e f \|^2_{\HS\e}&
= 
\suml_{i\in\I\e} \left\|(f\ie-f )\wt\phi\ie  +  f\ie G\ie  \phi\ie\right\|^2_{\L(Y\ie)}\\ \label{cond1:1}
&\leq
2\suml_{i\in\I\e} \|f\ie-f  \|^2_{\L(Y\ie)} +  
2\suml_{i\in\I\e}|f\ie|^2 \| G\ie \|^2_{\L(Y\ie)}.
\end{align}

The first term in the right-hand-side of \eqref{cond1:1} is  estimated via
the Poincar\'{e} inequality \eqref{Poincare}. Applying it for $g\ceq f-f\ie$, we  obtain
\begin{equation}\label{P:est}
  \suml_{i\in\I\e}
\|f-f\ie\|^2_{\L(\square\ie)}\leq  
 \pi^{-2}\eps^2 \suml_{i\in\I\e}
\|\nabla f\|^2_{\L(\square\ie)}\leq
 \pi^{-2}\eps^2 \|\nabla f\|^2_{\L(\Omega)}.
\end{equation}

Now, we estimate the second term. 
The Cauchy-Schwarz inequality yields
\begin{gather}\label{CaSw}
|f\ie|^2=\eps^{-2n}\left(\int_{\square\ie} f(x)\d x\right)^2\leq \eps^{-n}\|f\|^2 _{\L(\square\ie)}.
\end{gather}
Let us prove  that
\begin{gather}\label{L2Gie:final} 
\|G\ie \|^2_{\L(Y\ie)}\leq C\eps^{n+2}.
\end{gather} 
Indeed, straightforward calculation gives 
\begin{multline}\label{L2Gie:1}\hspace{-2ex}
 \|G\ie\|^2_{\L(Y\ie)}= 
 C V\e^2\eps^{2n}
 \cdot
\begin{cases}
\ds d\e^{4-n}-\left({\eps\over 2}\right)^{4-n},&n\ge 5,\\[1ex]
\ds \ln{\eps\over 2}-\ln d\e,&n=4,\\[1ex] 
\ds {\eps\over 2} -d\e ,&n= 3,\\[1ex]
\ds {\eps^2\over 4}\left(2 \ln^2{\eps\over 2}-2\ln{\eps\over 2} +1\right)-
 d\e^2\left(2 \ln^2 d\e -2\ln{d\e} +1\right),&n=2,
\end{cases}
\end{multline}
where $C>0$ depends only on $n$.
From \eqref{Vpq}, \eqref{pq+} we conclude
that $\sup_{\eps\in (0,\eps_0]}V\e Q\e^{-1/2}<\infty$ provided \eqref{PfinQfin} holds.
Hence (see the definition of $Q\e$) 
\begin{gather}\label{Ve:de}
V\e^2\leq  C 
d\e^{n-2}\eps^{-n}\text{ if }n\not= 2.
\end{gather}
Using \eqref{Ve:de} and  taking into account that 
$2d\e< {\eps}\le {\eps_0 }<1$, we deduce from \eqref{L2Gie:1}:
\begin{gather}\label{L2Gie:2a} 
\|G\ie \|^2_{\L(Y\ie)} \le C\eps^{n+2}\cdot
\begin{cases}
 \Lambda\e^2 ,&n\ge 5,\\ 
 \deps^2|\ln \deps| ,&n=4
\end{cases}
\end{gather}
(recall that $\Lambda\e$ is given in \eqref{Lambda}).
Moreover, since $\sup_{\eps\in (0,\eps_0]} V\e<\infty$, we also get  
\begin{gather}\label{L2Gie:2b} 
\|G\ie \|^2_{\L(Y\ie)}\leq C\eps^{n+2}\cdot
\begin{cases}
 \eps^2 ,&n=3,\\ 
 \eps^2|\ln\eps|^2 ,&n=2.
\end{cases}
\end{gather} 
Combining \eqref{L2Gie:2a}--\eqref{L2Gie:2b} and taking into account that 
$\eps<1$, $\Lambda\e< 1$ (and, consequently, $\eps|\ln\eps|<1$, $\Lambda\e|\ln\Lambda\e|<1$), we arrive at  the required estimate \eqref{L2Gie:final}
\footnote{In fact, since $\Lambda\e\to 0$ as $\eps\to 0$, one has even better asymptotics 
$\|G\ie \|^2_{\L(Y\ie)}=o(\eps^{n+2})$. 
However, it does not lead to an improvement of the resulting estimate \eqref{cond1:final} -- 
this is hindered by the inequality \eqref{P:est}. }.
Finally, by \eqref{CaSw} and \eqref{L2Gie:final}, we obtain the 
estimate for the second term on the right-hand-side of \eqref{cond1:1}:
\begin{gather}\label{Q:est}
\suml_{i\in\I\e}|f\ie|^2\|G\ie \|^2_{\L(Y\ie)}\leq
C\eps^2\suml_{i\in\I\e} \|f\|^2_{\L(\square\ie)}
\leq C\eps^2\|f\|^2_{\L(\Omega)}.
\end{gather}
 
It follows from \eqref{cond1:1}, \eqref{P:est}, \eqref{Q:est} that
\begin{gather}\label{cond1:final}
\|\J\e^1 f -\J\e f \|_{\HS\e}\leq C\eps\|f\|_{\HS^1}.
\end{gather}

\subsection{Check of condition \eqref{thA1:2}} \label{subsec:5:2}

Let $u\in\dom(\a\e)$.
 One has
\begin{gather}
\label{cond2:1}
\|\wt\J\e^1 u -\wt\J\e u \|^2_{\HS}=
\suml_{i\in\I\e}\|\wt\J\e^1 u\|^2_{\L(D\ie)}.
\end{gather}
Applying Lemma~\ref{lemma:MK06} for $D\ceq\square\ie$, $D_1\ceq D\ie$, $D_2\ceq\square\ie$, $g\ceq \wt\J\e u$, we get
\begin{align}\label{cond2:2}
\|\wt\J\e^1 u\|^2_{\L(D\ie)}
\leq C\left(\deps^n
\|\wt\J\e^1 u\|^2_{\L(\square\ie)} + \eps d\e
\|\nabla \wt\J\e^1 u\|^2_{\L(\square\ie)}\right).
\end{align}
Using  \eqref{wtJ1prop+}, \eqref{cond2:2}, 
we can conclude \eqref{cond2:1} as follows:
\begin{equation}\label{cond2:final} 
\|\wt\J\e u -\wt\J\e^1 u \|_{\HS\e}\leq
C \delta_{\eps,1} \|u\|_{\HS^1\e},
\end{equation}
where 
\begin{gather}\label{delta:1}
\delta_{\eps,1}\ceq \max\left\{\deps^{n/2};\, (\eps d\e)^{1/2}\right\}.
\end{gather}

\subsection{Check of condition~\eqref{thA1:3}}\label{subsec:5:3}

Let $u\in \dom(\a\e)$, $f\in \dom(\A)=\H^2(\Omega)\cap\H^1_0(\Omega)$. 
One has:
\begin{align} \notag
 \a\e[u,\J\e^1 f]-\a[\wt \J\e^1 u,f]=&
 -\underset{I_{\eps}^1\ceq }{\underbrace{\suml_{i\in\I\e}(\nabla \wt \J\e^1 u,\nabla f)_{\L(D\ie)}}}+
  \underset{I_{\eps}^2\ceq }{\underbrace{\suml_{i\in\I\e}\left(\nabla u,\nabla \left((f\ie-f)\wt\phi\ie\right)\right)_{\L(Y\ie)}}}\\&+
  \underset{I_{\eps}^3\ceq }{\underbrace{\suml_{i\in\I\e}(\nabla u,\nabla (f\ie G\ie \phi\ie))_{\L(Y\ie)}+
  \suml_{i\in\I\e}\ga\e\int_{\partial D\ie}
  u\, \overline{\J\e^1f}\d s
 -V(\wt \J\e^1 u,f)_{\L(\Omega)} }}.\label{4.1}
\end{align}
Below we estimate  each term on the right-hand-side of \eqref{4.1}. 

\subsubsection*{Estimate of $I_{\eps}^1$} 
One has using \eqref{wtJ1prop+}:
\begin{align}\notag
|I_{\eps}^1|^2\leq \suml_{i\in\I\e}\|\nabla (\wt \J\e^1 u)\|^2_{\L(D\ie)}\suml_{i\in\I\e}\|\nabla f\|^2_{\L(D\ie)}
&\leq 
\|\wt \J\e^1 u\|_{\H^1(\Omega)}^2\suml_{i\in\I\e}\|\nabla f\|^2_{\L(D\ie)}\\ 
&\leq
C\|u\|^2_{\HS^1\e}\suml_{i\in\I\e}\|\nabla f\|^2_{\L(D\ie)}.
\label{I1:1}
\end{align}
Recall that $\delta_{\eps,1}$ is defined in \eqref{delta:1}.
Applying Lemma~\ref{lemma:MK06} for 
$D\ceq\square\ie$, $D_1\ceq D\ie$, $D_2\ceq \square\ie$ and $g\ceq |\nabla f|$, one has (cf.~\eqref{cond2:2})
\begin{gather}\label{I1:2}
\suml_{i\in\I\e}\|\nabla f\|^2_{\L(D\ie)}\leq
C \delta_{\eps,1}^2
\suml_{i\in\I\e}\|f\|^2_{\H^2(\square\ie)}\le 
C \delta_{\eps,1}^2\|f\|^2_{\H^2(\Omega)}.
\end{gather}
Furthermore, uniform regularity of $\Omega$ yields the  estimate \cite[Theorem~2]{Br61}
\begin{gather}\label{el:reg}
\forall f\in\dom(\A):\quad 
\|f\|^2_{\H^2(\Omega)}\leq 
C\left(\|\A f\|^2_{ \HS}+\|f\|^2_{\HS}\right)\leq
C\|f\|^2_{\HS^2}
\end{gather}
(the last inequality in \eqref{el:reg} follows from $(\A f,f)_{\HS}\ge 0$).
Combining \eqref{I1:1}--\eqref{el:reg}, we get
\begin{gather} \label{I1:final}
|I_{\eps}^1|\leq C\delta_{\eps,1}\|f\|_{\HS^2}\|u\|_{\HS\e^1}.
\end{gather}

\subsubsection*{Estimate of $I_{\eps}^2$}

We denote
\begin{align}\label{Fie}
F\ie&\ceq
\left\{x\in\R^n:\ d\e<|x-x\ie|< \wt d\e\right\},\text{ where }
\wt d\e\ceq
\begin{cases}
2d\e,&n\ge 3,\\
\sqrt{\eps d\e},&n=2.
\end{cases}
\end{align}
It is easy to see that $\supp(\wt\phi\ie)\subset\overline{F\ie}$. Using this fact,
$\|\nabla u\|_{\L(\Omega\e)}\leq \|u\|_{\HS\e^1}$ and 
 $|\wt\phi\ie(x)|\leq 1$, we get 
\begin{align}  
|I_{\eps}^2|&=
\left|\suml_{i\in\I\e}\left[(\nabla u, \wt\phi\ie\nabla f)_{\L(F\ie)} + (\nabla u,(f-f\ie)\nabla \wt\phi\ie)_{\L(F\ie)}\right]\right| \notag \\ 
&\leq 
\|u\|_{\HS\e^1}\left[
\Bigl( {\suml_{i\in \I\e}\|\nabla f\|_{\L(F\ie)}^2\Bigr)^{1/2}}  + 
 {\suml_{i\in \I\e}\|(f-f\ie)\nabla \wt\phi\ie\|_{\L(Y\ie)}^2\Bigr)^{1/2}} \right].\label{I2:1}
\end{align} 
Using   Lemma~\ref{lemma:MK06}
with $D\ceq \square\ie$, $D_1\ceq F\ie$, $D_2\ceq \square\ie$ and $g\ceq |\nabla f|$,
we obtain
\begin{gather}\label{delta:2}
   \|\nabla f\|_{\L(F\ie)}
  \leq C \delta_{\eps,2}\|f\|_{\H^2(\square\ie)},\text{ where }
 \delta_{\eps,2}\ceq
\begin{cases}
 \delta_{\eps,1},&n\ge 3,\\
\max\left\{\deps^{1/2} ;\,\eps^{3/4} d\e^{1/4}\right\},&n=2,
\end{cases}
\end{gather}
where $\delta_{\eps,1}$ is given in \eqref{delta:1}. 
Taking into account \eqref{el:reg}, we get finally
\begin{gather}\label{sest:final}
\left(\suml_{i\in \I\e}\|\nabla f\|_{\L(F\ie)}\right)^{1/2} 
\leq 
C  \delta_{\eps,2} \|f\|_{\HS^2}.
\end{gather}

To proceed further we need the classical H\"older inequality 
$$\|FG\|_{\mathsf{L}^1(Y\ie)}\leq \|F\|_{\mathsf{L}^\mathbf{p}(Y\ie)}\|G\|_{\mathsf{L}^\mathbf{q}(Y\ie)},\quad
\forall 
\mathbf{p},\mathbf{q}\in[1,\infty],\ \mathbf{p}^{-1}+\mathbf{q}^{-1}=1$$
(for $\mathbf{p}=\infty$ one has the convention $\mathbf{p}^{-1}=0$).
Setting  $\mathbf{p}\ceq p/2$, $\mathbf{q}\ceq q/2$, $F\ceq |f-f\ie|^2$ and $G\ceq |\nabla \wt\phi\ie|^2$, we obtain
\begin{gather}\label{Hoelder}
\|(f-f\ie)\nabla \wt\phi\ie\|^2_{\L(Y\ie)}\leq 
\|f-f\ie\|^2_{\mathsf{L}^p(Y\ie)} \|\nabla \wt\phi\ie\|^2_{\mathsf{L}^q(Y\ie)}
,\quad \forall  p,q\in[2,\infty],\  {1\over p}+ {1\over q}={1\over 2}.
\end{gather} 
Recall that $\Lambda\e$ is given in \eqref{Lambda}; 
 due to \eqref{de},  $|\ln \Lambda\e|\ge \ln 4 >1$. We choose $p,\, q$ as follows:
\begin{align}\label{pmax}
&p=\frac {2n}{n-4} \text{\, if }n\geq 5,\quad
&&p=2|\ln \Lambda\e|\text{\, if }n=4,\quad
&&p=\infty\text{\, if }n=2,3,\\\label{qmax}
&q={n\over 2}\text{\, if }n\geq 5,\quad 
&&q={2\over 1-|\ln\Lambda\e|^{-1}}\text{\, if }n=4,\quad 
&&q=2\text{\, if }n=2,3
\end{align} 
(thus $p^{-1}+q^{-1}=1/2$, moreover,  $p$ satisfies \eqref{p}).
Using \eqref{sobolev} for $g\ceq f-f\ie$ and $p$ as in \eqref{pmax} and taking into account \eqref{lemma:c4p:est}, we obtain the estimate  
\begin{gather}\label{pmax+}
\|f-f\ie\|_{\mathsf{L}^p(\square\ie)}\leq C
\|f\|_{\mathsf{H}^2(\square\ie)}
\begin{cases}
\eps^{-1}&n\geq 5,\\
|\ln\Lambda\e|\cdot\eps^{2|\ln\Lambda\e|^{-1}-1}&n= 4,\\
\eps^{-1/2}&n= 3,\\
1&n= 2.\\
\end{cases}
\end{gather} 
Moreover, straightforward calculation gives
  \begin{gather}
    \label{wtphi-est}
      \|\nabla\wt\phi\ie\|_{\mathsf{L}^q(Y\ie)}\leq 
       C\begin{cases}
       d\e,&n\ge 5,\\
       d\e^{1-2|\ln\Lambda\e|^{-1}},&n=4,\\
       d\e^{1/2},&n=3,\\
       |\ln \deps|^{-1/2},&n=2,
    \end{cases} 
  \end{gather} 
where $q$ is defined by \eqref{qmax}.  
Using \eqref{Hoelder}, \eqref{pmax+}, \eqref{wtphi-est} and  
$\Lambda\e^{-2|\ln\Lambda\e|^{-1}}=\exp(2)$,
we get
\begin{gather}
 \|(f-f\ie)\nabla\wt\phi\ie\|_{\L(Y\ie)}\leq  C \delta_{\eps,3}
\|f\|_{\H^2(\square\ie)},\quad\text{where }
\delta_{\eps,3} \ceq 
\begin{cases}
\deps,&n\ge 5,
\\[1mm] \deps |\ln\Lambda\e|,&n=4,
\\[1mm] 
   \deps^{1/2},&n=3,
\\[1mm] 
 |\ln \deps|^{-1/2},&n=2.
\end{cases}
\label{delta:3}
\end{gather}
Taking into account \eqref{el:reg}, we get from \eqref{delta:3}:
\begin{gather}\label{test:final}
  \left(\suml_{i\in\I\e}\|(f-f\ie)\nabla\wt\phi\ie\|_{\L(Y\ie)}^2\right)^{1/2}\leq  C \delta_{\eps,3}
\|f\|_{\HS^2}.
\end{gather}  

Combining \eqref{I2:1}, \eqref{sest:final}, \eqref{test:final}, we arrive at the estimate
\begin{gather}\label{I2:final} 
|I\e^2|\leq 
C \max\left\{\delta_{\eps,2} ;\,\delta_{\eps,3} \right\}\|f\|_{\HS^2}\|u\|_{\HS\e^1}.
\end{gather}
 
\subsubsection*{Estimate of $I_{\eps}^3$}
It is easy to see that 
\begin{gather}
\label{prop1}
(\J\e^1 f)(x)= f\ie(G\ie(x)+1),\quad x\in\partial D\ie.
\end{gather}
Moreover, since  
$$
(\phi\ie-1)\restriction_{\partial D\ie} =\phi\ie\restriction_{\partial Y\ie\setminus\partial D\ie}=0,\qquad
{\partial \phi\ie\over\partial n}\restriction_{\partial D\ie}={\partial \phi\ie\over\partial n}\restriction_{\partial Y\ie\setminus\partial D\ie}=0, 
$$
one has 
\begin{align}\label{prop2}
&{\partial (f\ie G\ie\phi\ie)\over \partial n}= f\ie {\partial G\ie \over \partial n}& &\text{on}\quad\partial D\ie,
\\\label{prop3}
&{\partial (f\ie G\ie\phi\ie)\over \partial n}= 0& &\text{on}\quad\partial Y\ie\setminus\partial D\ie.
\end{align}
Using \eqref{prop1}--\eqref{prop3} and integrating by parts we get
\begin{align*}
I_{\eps}^3
&=\suml_{i\in\I\e}(\nabla u, \nabla (f\ie G\ie\phi\ie) )_{\L(Y\ie)}+
  \suml_{i\in\I\e}\ga\e\int_{\partial D\ie}
  u\, \overline{f\ie}(G\ie+1)\d s
 -V(\wt \J\e^1 u,f)_{\L(\Omega)}\\
&=-\suml_{i\in\I\e}(u, \Delta (f\ie G\ie\phi\ie) )_{\L(Y\ie)}+
  \suml_{i\in\I\e}\int_{\partial D\ie}
  \overline{f\ie}\left({\partial  G\ie \over \partial n}+ \ga\e (G\ie +1)\right)u\d s
-V(\wt \J\e^1 u,f)_{\L(\Omega)}.
\end{align*}
Straightforward calculation, using \eqref{peqe}, \eqref{Ve}, \eqref{Gie}, yields
\begin{gather}\label{G:bk}
\begin{array}{rl}\ds
{\partial G\ie\over \partial n} +\gamma\e (G\ie+1)=0&\text{on } \partial  D\ie,
\end{array}
\end{gather}
whence
\begin{align}\label{I3:1}
I\e^3=-\suml_{i\in\I\e}\big(u, \Delta (f\ie G\ie\phi\ie)  \big)_{\L(Y\ie)}
-V(\wt \J\e^1 u,f)_{\L(\Omega)}. 
\end{align} 
It is easy to see that
\begin{gather}
\label{intGie}
-\int_{\partial D\ie}{\partial G\ie\over \partial n} \d s= V\e \eps^n.
\end{gather}
We denote by $u\ie$ the mean value of the function $(\wt\J\e^1 u)(x)$ in the domain $\square\ie$, i.e.,
$$u\ie=\eps^{-n}\int_{\square\ie}(\wt\J\e^1 u)(x)\d x.$$
Using \eqref{prop2}, \eqref{prop3}, \eqref{intGie} and Gauss's divergence theorem, we can transform \eqref{I3:1} as follows:
\begin{align*}
I\e^3
&= -\suml_{i\in\I\e}\big(u\ie, \Delta (f\ie G\ie\phi\ie)  \big)_{\L(Y\ie)}
-V(\wt \J\e^1 u,f)_{\L(\Omega)}
-\suml_{i\in\I\e}\big(u-u\ie, \Delta (f\ie G\ie\phi\ie)  \big)_{\L(Y\ie)}\\
\notag
&=-\suml_{i\in\I\e}u\ie \overline{f\ie} \int_{\partial D\ie}{\partial G\ie\over \partial n} \d s
-V(\wt \J\e^1 u,f)_{\L(\Omega)}
-\suml_{i\in\I\e}\big(u-u\ie, \Delta (f\ie G\ie\phi\ie)  \big)_{\L(Y\ie)}\\
\notag
&=\underset{I\e^{3,1}\ceq }{\underbrace{\suml_{i\in\I\e}u\ie \overline{f\ie} V\e\eps^n
-V\suml_{i\in\I\e}(\wt \J\e^1 u,f)_{\L(\square\ie)}}}
-\underset{I_{\eps}^{3,2}\ceq }{\underbrace{\suml_{i\in\I\e}\big(u-u\ie, \Delta (f\ie G\ie\phi\ie)  \big)_{\L(Y\ie)}}}\\
&
-\underset{I_{\eps}^{3,3}\ceq }{\underbrace{V(\wt \J\e^1 u,f)_{\L(\Omega\setminus\cup_{i\in\I\e}\square\ie)}}},
\end{align*}

Evidently, one has
$$I\e^{3,1}=
\suml_{i\in\I\e} (V\e-V) u\ie \overline{f\ie} \eps^n -  
\suml_{i\in\I\e}V (\wt \J\e^1 u,f-f\ie)_{\L(\square\ie)}.$$
Then, using \eqref{Poincare}, \eqref{CaSw} and similarly the Cauchy-Schwarz inequality for $u\ie$, we get
\begin{align*}
|I\e^{3,1}|&=
|V\e-V| 
\left(\suml_{i\in\I\e}|f\ie|^2\eps^n\right)^{1/2}
\left(\suml_{i\in\I\e}|u\ie|^2\eps^n\right)^{1/2}
\\ 
&+
V\left(\suml_{i\in\I\e}\|f-f\ie\|^2_{\L(\square\ie)}\right)^{1/2}\left(\suml_{i\in\I\e}\|\wt \J\e^1 u\|^2_{\L(\square\ie)}\right)^{1/2}\\
&\leq |V\e-V|\|f\|^2_{\L(\Omega)}\|\wt \J\e^1 u\|^2_{\L(\Omega)}+
V\eps\|\nabla f\|^2_{\L(\Omega)}\|\wt \J\e^1 u\|^2_{\L(\Omega)},
\end{align*}
whence, taking into account \eqref{wtJ1prop+}, we conclude
\begin{gather}\label{I31:final}
|I\e^{3,1}|\le C\max\{|V\e-V|;\, \eps\}\|f\|_{\HS^1}\|u\|_{\HS^1\e}.
\end{gather}

We denote
\begin{gather}\label{wtY}
\wt Y\ie\ceq \left\{x\in\R^n:\ {\eps\over 4}<|x|<{\eps\over 2}\right\}.
\end{gather}
One has $\supp (D^\al\phi\ie)\subset\overline{\wt Y\ie}$, $|\al|=1,2$.
Using this and the fact that $\Delta G\ie=0$ for $x\not= x\ie$, we estimate the term $I\e^{3,2}$ as follows:
\begin{align*}
|I\e^{3,2}|&=\left|
\suml_{i\in\I\e}\big(u-u\ie, 2f\ie \nabla G\ie\cdot\nabla \phi\ie + f\ie  G\ie\Delta \phi\ie \big)_{\L(Y\ie)}\right|
\leq \left(\suml_{i\in\I\e}\|u-u\ie\|^2_{\L(Y\ie)}\right)^{1/2}
\\
&\times\left(\suml_{i\in\I\e}2|f\ie|^2\left(\|2\nabla G\ie\cdot\nabla\phi\ie\|^2_{\L(\wt Y\ie)}+ \|G\ie\Delta\phi\ie\|^2_{\L(\wt Y\ie)}\right)\right)^{1/2}.
\end{align*}
Using \eqref{Poincare} and \eqref{wtJ1prop+},  we get
\begin{equation}\label{uPoincare}\hspace{-1ex}
 \sum_{i\in\I\e}\|u-u\ie\|^2_{\L(Y\ie)}\leq
 \sum_{i\in\I\e}\|\wt\J\e^1 u-u\ie\|^2_{\L(\square\ie)}\leq
C\eps^2 \sum_{i\in\I\e}\|\nabla (\wt\J\e^1 u)\|^2_{\L(\square\ie)}\leq
C\eps^2\|u\|^2_{\HS\e^1}.
\end{equation}
Furthermore, taking into account \eqref{V+}, one can  easily get the pointwise estimate 
\begin{gather}\label{pointwiseYie}
\text{for }x\in  \wt Y\ie:
\quad
\begin{array}{ll}
|G\ie(x)|\leq 
C \eps \delta_{\eps,4},
&\quad
|\nabla G\ie(x)|\leq 
C \eps,\\[5mm]
|\Delta\phi\ie(x)|\leq 
C \eps^{-2},&\quad
|\nabla \phi\ie(x)|\leq 
C \eps^{-1},
\end{array}
\end{gather}
where
\begin{gather}
\label{delta:4}
\delta_{\eps,4} \ceq 
\begin{cases}
\eps,&n\ge 3,\\ 
 \eps|\ln\eps|,&n=2
\end{cases}
\end{gather}
Using \eqref{CaSw}, \eqref{pointwiseYie} and   $|\wt Y\ie|\leq C\eps^n$, one gets
\begin{gather}\label{long:est}
|f\ie|^2\left(\|2\nabla G\ie\cdot\nabla\phi\ie\|^2_{\L(\wt Y\ie)}+ \|G\ie\Delta\phi\ie\|^2_{\L(\wt Y\ie)}\right)
\leq C(1+\eps^{-2}(\delta_{\eps,4})^2)\|f\|^2_{\L(\square\ie)}.
\end{gather}
Taking into account 
\begin{gather}\label{delta:4:rem}
\eps\leq |\ln\eps_0|^{-1} \delta_{\eps,4}\text{ as }n=2
\end{gather}
(recall that $\eps\leq\eps_0<1$)
we infer from  \eqref{uPoincare}, \eqref{long:est} that
\begin{align} 
\label{I32:final}
|I\e^{3,2}|\leq 
C\delta_{\eps,4} \|f\|_{\HS}\|u\|_{\HS^1}
\end{align}

Finally, using Lemma~\ref{lemma:layer} (for $g\ceq f$ and  $g\ceq \J\e^1 u$) and \eqref{wtJ1prop+},  one gets the estimate for $I\e^{3,3}$:
\begin{gather}\label{I33:final}
|I\e^{3,3}|\leq  
C\eps^2\|\nabla f\|_{\L(\Omega)}\|\nabla (\wt\J\e^1 u)\|_{\L(\Omega)}\leq
C\eps^2\|f\|_{\HS^1}\|u\|_{\HS^1\e},
\end{gather} 
and, as a result (see \eqref{I31:final}, \eqref{I32:final}, \eqref{I33:final}), we arrive at the estimate for 
$I\e^3$:
\begin{gather}\label{I3:final}
|I\e^3|
\leq C\max\{|V\e-V|;\,\eps;\,\delta_{\eps,4};\,\eps^2  \}\|f\|_{\HS^1}\|u\|_{\HS^1\e}
\leq
C\max\{|V\e-V|;\, \delta_{\eps,4}  \}\|f\|_{\HS^1}\|u\|_{\HS^1\e}
\end{gather} 
(in the last inequality in \eqref{I3:final} we use $\eps<1$ and \eqref{delta:4:rem}).

Combining \eqref{I1:final}, \eqref{I2:final},  \eqref{I3:final} and taking into account \eqref{scale:est}, 
we get 
\begin{equation}\label{cond3:final}
|\a\e[u,\J\e^1 f]-\a[\wt \J\e^1 u,f]|\leq
C\max\left\{\delta_{\eps,1};\,\delta_{\eps,2};\,\delta_{\eps,3};\,\delta_{\eps,4};\,|V\e-V|\right\} \|f\|_{\HS^2}\|u\|_{\HS^1\e}.
\end{equation} 

\subsection{Conclusion}\label{subsec:5:4}

Combining \eqref{cond1:final}, \eqref{cond2:final}, \eqref{cond3:final} 
and taking into account \eqref{delta:4:rem}, we conclude 
that the
conditions \eqref{thA1:1}--\eqref{thA1:3} of Theorem~\ref{thA1} hold with
$$ 
\delta\e\ceq C
\max\left\{\delta_{\eps,1};\,\delta_{\eps,2};\,\delta_{\eps,3};\,\delta_{\eps,4};\,|V\e-V|\right\},
$$
where $\delta_{\eps,k}$, $k=1,2,3,4$ are defined in
\eqref{delta:1}, \eqref{delta:2}, \eqref{delta:3}, \eqref{delta:4}, respectively.
It is straightforward to check that  
\begin{gather}\label{delta:eta}
\delta\e\leq C\eta\e,
\end{gather}
where $\eta\e$ is defined 
by \eqref{eta}, Q.E.D.

\begin{remark}\label{remark:smooth}
Recall that $\partial\Omega$ is assumed to be uniformly regular \cite{Br61}.
We need it to guarantee the fulfillment of \eqref{el:reg}.
However, it is well-known that \eqref{el:reg} remains valid under less restrictive assumptions on $\Omega$, for example, if $\partial\Omega$ is compact and belongs to the $\mathsf{C}^{1,1}$ class or if $\Omega$ is a convex domain with Lipschitz boundary  \cite[Theorems~2.2.2.3 and 3.2.1.2]{Gr85}.
\end{remark}

\section{Proof of the main results}\label{sec:6}

\subsection{Proof of Theorem~\ref{th1}}\label{subsec:6:1}

In Section~\ref{sec:5} we established  the fulfillment of the properties \eqref{thA1:0}--\eqref{thA1:3}
with $\delta\e\le C\eta\e$. Then, by Theorem~\ref{thA1}, we immediately
arrive at the required estimate \eqref{th1:est}. Theorem~\ref{th1} is proven.

\subsection{Proof of Theorem~\ref{th2}}\label{subsec:6:1}
Let $g\in\L(\Omega)$, and $f\ceq (\A+\Id)^{-1} g$.
One has  
\begin{align}\notag
&\|(\A\e+\Id)^{-1}\J\e g  - \J\e(\A+\Id)^{-1} g \|_{\H^1(\Omega\e)}
\\ \notag
&\qquad\leq
\|(\A\e+\Id)^{-1}\J\e g  - \J\e^1(\A+\Id)^{-1} g \|_{\H^1(\Omega\e)}+
\|(\J\e^1 - \J\e)(\A+\Id)^{-1} g \|_{\H^1(\Omega\e)}\\  
&\qquad\leq
\|(\A\e+\Id)^{-1}\J\e g  - \J\e^1(\A+\Id)^{-1} g \|_{\HS\e^1}+
\|(\J\e^1 - \J\e)f\|_{\H^1(\Omega\e)}.\label{HL:1} 
\end{align} 
In Section~\ref{sec:5} we proved that \eqref{thA1:0}--\eqref{thA1:3} hold with $\delta\e\leq C\eta\e$. 
Hence by Theorem~\ref{thA1+} we have
\begin{align}
\|(\A\e+\Id)^{-1}\J\e g  - \J\e^1(\A+\Id)^{-1} g \|_{\HS\e^1}\leq C\eta\e\|g\|_{\L(\Omega)}.\label{HL:1+} 
\end{align} 

Now, we estimate the second term in the right-hand-side of \eqref{HL:1+}. 
Recall that $Y\ie $ is given in \eqref{Yie} and the inclusions \eqref{supp:phi} hold.
One has: 
\begin{gather}\label{HL:2}
 \|(\J\e^1 - \J\e)f \|_{\H^1(\Omega\e)}^2\leq 
2\suml_{i\in\I\e}\|(f-f\ie)\wt\phi\ie\|^2_{\H^1(Y\ie)}+
2\suml_{i\in\I\e}\|f\ie G\ie \phi\ie\|^2_{\H^1(Y\ie)}.
\end{gather} 
Using \eqref{P:est}, \eqref{Q:est} and $|\phi\ie|\leq 1$, $|\wt\phi\ie|\leq 1$,  we get
\begin{align}\label{HL:3a}
\suml_{i\in\I\e}\|(f-f\ie)\wt\phi\ie\|^2_{\L(Y\ie)}
&\leq \suml_{i\in\I\e}\|f-f\ie \|^2_{\L(\square\ie)}\leq 
C\eps^2\suml_{i\in\I\e}\|\nabla f\|^2_{\L(\square\ie)}\leq 
C\eps^2\|f\|^2_{\HS^1},\\\label{HL:3b}
\suml_{i\in\I\e}\|f\ie G\ie \phi\ie\|^2_{\L(Y\ie)}&\leq
\suml_{i\in\I\e}|f\ie|^2\|G\ie\|^2_{\L(Y\ie)} \leq C\eps^2
\suml_{i\in\I\e}\|f\|^2_{\L(\square\ie)}
\le
C\eps^2\|f\|^2_{\HS^1}.
\end{align}
Moreover, one has (see~\eqref{sest:final}, \eqref{test:final} and note
$\supp(\wt\phi\ie)\subset\overline{F\ie}$, where the set $F\ie$ is defined in \eqref{Fie})
\begin{align}\notag
\suml_{i\in\I\e}\|\nabla((f-f\ie)\wt\phi\ie)\|^2_{\L(Y\ie)}
&\leq 
2\suml_{i\in\I\e}\|\nabla f\|^2_{\L(F\ie)}+
2\suml_{i\in\I\e}\| (f-f\ie)\nabla \wt\phi\ie)\|^2_{\L(Y\ie)}
\\&\le
C \max\left\{(\delta_{\eps,2})^2;\,(\delta_{\eps,3})^2 \right\}\|f\|^2_{\HS^2},\label{HL:4}
\end{align}
where $\delta_{\eps,2} $, $\delta_{\eps,3} $ are given in \eqref{delta:2}, \eqref{delta:3}, respectively.
Finally,
using  \eqref{CaSw} and taking into account that
$|\phi\ie|\leq 1$, $|\nabla\phi\ie|\leq C\eps^{-1}$, we obtain
\begin{align} 
\suml_{i\in\I\e}\|\nabla(f\ie G\ie \phi\ie)\|^2_{\L(Y\ie)}
&\leq C\eps^{-n}\suml_{i\in\I\e}\|f\|_{\L(\square\ie)}^2\left(\|\nabla G\ie \|^2_{\L(Y\ie)}+
  \eps^{-2}\| G\ie \|^2_{\L(\wt Y\ie)}\right),
\label{HL:5}
\end{align}
where $\wt Y\ie$ is given in \eqref{wtY}; here we have used that 
$\supp(\nabla\phi\ie)\subset\overline{\wt Y\ie}$.
Straightforward calculation yields
\begin{gather}
\begin{array}{ll}
\|\nabla G\ie \|^2_{\L(Y\ie)}\le \|\nabla G\ie\|^2_{\L(\R^n\setminus\overline D\ie)}={V\e^2 Q\e^{-1}}\eps^n,&n\ge 3,\\[2mm]
\|\nabla G\ie \|^2_{\L(Y\ie)}\le \|\nabla G\ie\|^2_{\L(B_1(x\ie)\setminus\overline D\ie)}={V\e^2 Q\e^{-1}}\eps^2,&n=2,
\end{array}\label{nablaGie}
\end{gather} 
where $B_1(x\ie)$ is the unit ball with the center at $x\ie$.
Moreover, using the pointwise estimates  \eqref{pointwiseYie} and $|\wt Y\ie|\leq C\eps^n$, we get
\begin{gather}
\begin{array}{ll}
\| G\ie\|^2_{\L(\wt Y\ie)}\le C\eps^{n+2}(\delta_{\eps,4})^2,
\end{array}\label{L2Gie:3}
\end{gather}
where $\delta_{\eps,4}$ is given in \eqref{delta:4}.
It follows from \eqref{HL:5}--\eqref{L2Gie:3} that
\begin{gather}\label{HL:6}
\suml_{i\in\I\e}\|\nabla(f\ie G\ie \phi\ie)\|^2_{\L(Y\ie)}\leq 
C\max\left\{{V\e^2 Q\e^{-1}};\,(\delta_{\eps,4})^2\right\} \|f\|^2_{\L(\Omega)}.
\end{gather}
Combining \eqref{HL:2}--\eqref{HL:4}, \eqref{HL:6} and taking into account \eqref{scale:est}, \eqref{delta:eta}
and $\eps\leq C\delta_{\eps,4}$,
we arrive at  
\begin{align} \notag
 \|(\J\e^1 - \J\e)f \|_{\H^1(\Omega\e)}&\leq 
  C \max\left\{V\e Q\e^{-1/2};\, \delta_{\eps,2} ;\, \delta_{\eps,3} ;\,\delta_{\eps,4} \right\}\|f\|^2_{\HS^2}\\&\leq
  C \max\left\{V\e Q\e^{-1/2};\,\eta\e\right\}\|f\|_{\HS^2}=
  C \max\left\{V\e Q\e^{-1/2};\,\eta\e\right\}\|g\|_{\L(\Omega)}.\label{HL:7}
\end{align}

The required estimate \eqref{th2:est} follows from
 \eqref{HL:1}, \eqref{HL:1+}, \eqref{HL:7} and the definition \eqref{eta:prime:1} of $\eta\e'$.

\subsection{Proof of Theorem~\ref{th3}}\label{subsec:6:3}
Let $g\in\L(\Omega)$, and $f\ceq (\A+\Id)^{-1} g$.
One has
\begin{multline} 
\|(\A\e+\Id)^{-1}\J\e g  - (\Id+G\e)\J\e(\A+\Id)^{-1} g \|_{\H^1(\Omega\e)}
\\   
\leq  
\|(\A\e+\Id)^{-1}\J\e g  - \J\e^1(\A+\Id)^{-1} g \|_{\HS^1\e}+
\|(\J\e^1 - (\Id+G\e)\J\e))f \|_{\H^1(\Omega\e)}.
\label{corrector:1} 
\end{multline}
Again, using Theorem~\ref{thA1+} and taking into account \eqref{eta:eta}, we have
\begin{gather}
\label{corrector:1a}
\|(\A\e+\Id)^{-1}\J\e g  - \J\e^1(\A+\Id)^{-1} g \|_{\HS^1\e}\leq C\wt\eta\e\|g\|_{\L(\Omega)},
\end{gather}
where $\wt\eta\e$ is defined by \eqref{wtdeltan}.

Now, we estimate the second term on the right-hand-side of \eqref{corrector:1}.
One has
\begin{align}\label{corrector:1b}
(\J\e^1 - (\Id+G\e)\J\e))f=
\suml_{i\in\I\e}(f\ie-f)\wt\phi\ie
+
\suml_{i\in\I\e}(f\ie-f)G\ie\phi\ie.
\end{align} 
Using \eqref{eta:eta}, \eqref{delta:eta},
\eqref{HL:3a},  \eqref{HL:4}, we get the estimate 
\begin{align}\notag
\Bigl\|\suml_{i\in\I\e}(f-f\ie)\wt\phi\ie\Bigr\|_{\H^1(\Omega\e)}
=
\left(\suml_{i\in\I\e}\|(f-f\ie)\wt\phi\ie\|_{\H^1(Y\ie)}^2\right)^{1/2}
&\leq
C\max\{\eps;\, \delta_{\eps,2};\, \delta_{\eps,3}\}\|f\|_{\HS^2}
\\ \label{corrector:1c}
&\le C\wt\eta\e\|f\|_{\HS^2}=
C\wt\eta\e\|g\|_{\HS}.
\end{align}
To proceed further, we observe that
\begin{gather}\label{Linfty}
\left.
\begin{array}{ll}
n\geq 3:& \|G\ie\|_{\mathsf{L}^\infty(\R^n\setminus D\ie)}\\[2mm]
n= 2:   & \|G\ie\|_{\mathsf{L}^\infty(B_1(x\ie)\setminus D\ie)}
\end{array}\right\}
\leq \|G\ie\|_{\mathsf{L}^\infty(\partial D\ie)}={V\e Q\e^{-1}}
={P\e\over P\e+Q\e}
\leq 1.
\end{gather}
Using \eqref{Poincare}, \eqref{Linfty}, and $|\phi\ie|\le 1$, 
we get
\begin{equation}\hspace{-1ex}
\Big\|\suml_{i\in\I\e}(f\ie-f)G\ie\phi\ie\Big\|_{\L(\Omega\e)}\leq 
 \left(\suml_{i\in\I\e}\|f\ie-f\|_{\L(\square\ie)}^2\right)^{1/2}\leq
C\eps\|\nabla f\|_{\L(\Omega)}\leq C\wt\eta\e\|f\|_{\HS^1}.
\label{corrector:2} 
\end{equation} 
Furthermore, we have
\begin{equation*}
\Bigl\|\nabla\Bigl(\suml_{i\in\I\e}(f\ie-f)G\ie\phi\ie\Bigr)\Bigr\|_{\L(\Omega\e)}
\leq
\underbrace{\left(\suml_{i\in\I\e}\| G\ie\phi\ie \nabla f\|^2_{\L(Y\ie)}\right)^{1/2}}_{\mathcal{I}\e^1\ceq}
+
\underbrace{\left(\suml_{i\in\I\e}\|(f-f\ie)\nabla(G\ie\phi\ie) \|^2_{\L(Y\ie)}\right)^{1/2}}_{\mathcal{I}\e^2\ceq}.
\end{equation*}

\subsubsection*{Estimate of $\mathcal{I}\e^1$}
One has (cf.~\eqref{Hoelder}):
\begin{gather}\label{corrector:3}
\mathcal{I}\e^1\le 
\left(\suml_{i\in\I\e}\|\nabla f\|^2_{\mathsf{L}^p(Y\ie)}
\|G\ie\phi\ie \|^2_{\mathsf{L}^q(Y\ie)}\right)^{1/2},\ p,q\in[2,\infty],\ {1\over p}+{1\over q}={1\over 2}.
\end{gather}
 We choose  $p$ and $q$ as follows:
\begin{gather*}
p=\ds\frac {2n}{n-2},\ q=n\text{ if }n\geq 3
\quad\text{and}\quad
p=6,\ q=3 \text{ if }n=2,
\end{gather*}
that is $p^{-1}+q^{-1}=1/2$, moreover, $p$ satisfies \eqref{p:grad}. 
Then, using the estimate \eqref{sobolev:grad}, we get
\begin{gather}\label{pmax+:grad}
\|\nabla f\|_{\mathsf{L}^p(\square\ie)}\leq C
\|f\|_{\mathsf{H}^2(\square\ie)}
\begin{cases}
\eps^{-1}&n\geq 3,\\
\eps^{-2/3}&n= 2.
\end{cases}
\end{gather} 
Moreover, using \eqref{Qpositive} and $|\phi\ie|\le 1$, we obtain via direct calculation (see the similar calculations in \eqref{L2Gie:1}, where $q=2$):
\begin{gather}\label{wtphi-est:grad}
\|G\ie\phi\ie\|_{\mathsf{L}^q(Y\ie)}\le 
\|G\ie\|_{\mathsf{L}^q(Y\ie)}
 \leq  C
 \begin{cases}
       \eps^{n/(n-2)},&n\ge 4,\\
       \eps^3|\ln\eps|^{1/3},&n=3,\\
       \eps^{8/3}|\ln\eps|  ,&n=2.
    \end{cases} 
  \end{gather}
Combining \eqref{corrector:3}, \eqref{pmax+:grad}, \eqref{wtphi-est:grad},
we arrive at the estimate
\begin{gather}\label{mathI1:final}
\mathcal{I}^1\e\leq C\|f\|_{\H^2(\Omega)}
\begin{cases}
       \eps^{2/(n-2)},&n\ge 4,\\
       \eps^{2}|\ln\eps |^{1/3},&n=3,\\
       \eps^{2}|\ln\eps| ,&n=2.
    \end{cases} 
\end{gather}

\subsubsection*{Estimate of $\mathcal{I}\e^2$}
One has (cf.~\eqref{Hoelder}):
\begin{gather}\label{corrector:3+}
\mathcal{I}\e^2\le 
\left(\suml_{i\in\I\e}\|f-f\ie\|^2_{\mathsf{L}^p(Y\ie)}
\|\nabla (G\ie\phi\ie) \|^2_{\mathsf{L}^q(Y\ie)}\right)^{1/2},\ p,q\in[2,\infty],\ 
{1\over p}+{1\over q}={1\over 2}.
\end{gather}
We choose $p$ and $q$ by \eqref{pmax} and \eqref{qmax}, respectively. 
The first factor on the right-hand-side of \eqref{corrector:3+} is already estimated in \eqref{pmax+}.
The second factor can be estimated via simple calculation (taking into account \eqref{eta:eta}):
\begin{gather}\label{corrector:4}
\|\nabla (G\ie\phi\ie) \|_{\mathsf{L}^q(Y\ie)}\leq 
C\left(\|\nabla G\ie \|_{\mathsf{L}^q(Y\ie)}+\eps^{-1}\| G\ie \|_{\mathsf{L}^q(\wt Y\ie)}\right)\leq C
\begin{cases}
\eps^{n/(n-2)},&n\ge 5,\\
\eps^{2-4|\ln \deps |^{-1}},&n= 4,\\
\eps^{3/2},&n= 3,\\
\eps,&n= 2
\end{cases}
\end{gather}
(recall that $\wt Y\ie$ is define by \eqref{wtY}, $\supp(\nabla \phi\ie)\subset\overline{\wt Y\ie}$,  and
$|\nabla\phi\ie|\leq C\eps^{-1}$).
Combining \eqref{el:reg}, \eqref{pmax+}, \eqref{corrector:3+}, \eqref{corrector:4},    
and taking into account that $
\eps^{-2|\ln\Lambda\e|^{-1}}\leq C
$ as $n=4$ (this estimate follows from \eqref{eta:eta}), 
we get
\begin{gather}\label{mathI2:final}
\mathcal{I}^2\e\leq C\|f\|_{\HS^2}
\begin{cases}
       \eps^{2/(n-2)},&n\ge 5,\\
         \eps|\ln\eps|,&n=   4,\\
       \eps,&n=2,3.
    \end{cases}
\end{gather}
It follows from  \eqref{mathI1:final}, \eqref{mathI2:final} that
\begin{gather}\label{corrector:5}
\big\|\nabla(\suml_{i\in\I\e}(f\ie-f)G\ie\phi\ie)\big\|_{\L(\Omega\e)}\leq C\wt\eta\e \|f\|_{\HS^2}=
C\wt\eta\e\|g\|_{\L(\Omega)}
\end{gather} 
(in the case $n=2$ we use $\eps\leq C\eps|\ln\eps|$ with $C=|\ln\eps_0|^{-1}$).

Combining \eqref{corrector:1}--\eqref{corrector:1c}, \eqref{corrector:2}, \eqref{corrector:5},
we arrive at the required estimate \eqref{th3:est}.

\subsection{Proof of Theorem~\ref{th4}}\label{subsec:6:4}

Let us apply Theorem~\ref{thA2}.
We have been  already proven  that the estimate \eqref{thA2:1}  holds with 
$\rho\e=\wt\rho\e= C\eta\e$. 
Moreover, due to \eqref{cond0:final}, one has $\J\e^*=\wt\J\e$ and thus 
$$
\left((\A\e+\Id)^{-1} \J\e-\J\e(\A+\Id)^{-1}\right)^*=
\wt\J\e(\A\e+\Id)^{-1} - (\A+\Id)^{-1}\wt\J\e,
$$
whence the estimate \eqref{thA2:2} also holds with $\wt\rho\e = C\eta\e$.

Since $\|u\|_{\L(\Omega\e)}=\|\wt\J\e u\|_{\L(\Omega)}$, we conclude that the estimate
\eqref{thA2:4} holds with $\wt\mu\e=1$, $\wt\nu\e=0$.
Finally, one has
\begin{gather}\label{spec:1}
\|f\|^2_{\HS}=\|\J\e f\|^2_{\HS\e}+\suml_{i\in\I\e}\|f\|^2_{\L(D\ie)}.
\end{gather}
Using Lemma~\ref{lemma:MK06} (for
$D\ceq \square\ie$, $D_1\ceq D\ie$, $D_2\ceq \square\ie\setminus \overline{D\ie}$, $g\ceq f$), 
we get 
\begin{align}\notag
\suml_{i\in\I\e}\|f\|_{\L(D\ie)}^2
&\leq
C\suml_{i\in\I\e}\left(\Lambda\e^n\|f\|_{\L(\square\ie\setminus \overline{D\ie})}^2+ \eps d\e 
\|\nabla f\|_{\L(\square\ie)}^2 \right)\\
&\label{spec:2}\leq 
C\left(\Lambda\e^n\|J\e f\|_{\HS\e}^2+
\eps d\e\cdot \a[f,f]\right)
\end{align}
(here we use $|\square\ie\setminus  {D\ie}|\ge C\eps^n$).
It follows from \eqref{spec:1}--\eqref{spec:2} that
\begin{gather*}\label{spec:3}
\|f\|^2_{\HS}\leq (1+C\Lambda\e^n)\|\J\e f\|^2_{\HS\e}+C\eps d\e\cdot \a[f,f],
\end{gather*}
i.e., the estimate \eqref{thA2:3} holds with $\mu\e\ceq 1+C\Lambda\e^n$ and $\nu\e\ceq C\eps d\e$.
Then,
applying Theorem~\ref{thA2}, we arrive at  
\begin{gather}\label{spec:4}
\widetilde{d_H}\left(\sigma(\A\e),\,\sigma(\A)\right)\leq 
C \left({\eps d\e}+
\sqrt{(\eps d\e)^2+\eta\e^2(1+C\Lambda\e^n)}
\right).
\end{gather}
Evidently, $\eps d\e <\eta\e$; also, we have $\Lambda\e< 1$. Hence \eqref{spec:4} implies the required estimate \eqref{th4:est}.

\subsection{Proof of Theorem~\ref{th5}}\label{subsec:6:5}

Let $v\in\L(\Omega\e)$. We set $u\ceq (\A\e+\Id)^{-1}v$. 
Lemma~\ref{lemma:crucial} yields
\begin{align}\label{infty:2}
\suml_{i\in\I\e}\|u\|_{\L(Y\ie)}^2\leq
C\max\left\{P\e^{-1},Q\e^{-1}\right\}\a\e[u,u]
\end{align}
(recall that $Y\ie$ is given in \eqref{Yie}).
We denote $R\ie\ceq \square\ie\setminus\overline{D\ie\cup Y\ie}.$
Using Lemma~\ref{lemma:MK06} for
$D\ceq \square\ie$, $D_1\ceq R\ie$, $D_2\ceq Y\ie$, $g\ceq \wt\J\e^1 u$ (note that $|R\ie|\leq C\eps^n$, $|Y\ie|\geq C\eps^n$) and \eqref{infty:2}, we obtain
\begin{align}\notag
\suml_{i\in\I\e}\|u\|_{\L(R\ie)}^2&\leq
C\suml_{i\in\I\e}\left(\|u\|_{\L(Y\ie)}^2+\eps^2
\|\nabla (\wt\J\e^1 u)\|_{\L(\square\ie)}^2 \right)\\ \label{infty:1}
&
\leq C\max\left\{P\e^{-1},Q\e^{-1}\right\}\a\e[u,u]+
C \eps^2 \|\nabla (\wt\J\e^1 u)\|_{\L(\Omega)}^2.
\end{align}  
Finally,  by 
Lemma~\ref{lemma:layer} applied for $g\ceq \wt\J\e^1 u$ we have
\begin{gather}\label{infty:3}
\|u\|_{\L(\Omega\setminus\overline{\cupl_{i\in\I\e}\square\ie})}
\leq 
C\eps\|\nabla (\wt J\e^1 u)\|_{\L(\Omega)}.
\end{gather}
Combining \eqref{wtJ1prop+}, \eqref{infty:2}--\eqref{infty:3},  we  get
\begin{align}\notag\hspace{-2ex}
\|(\A\e+\Id)^{-1}v\|_{\L(\Omega\e)}^2=
\|u\|_{\L(\Omega\e)}^2&=
\suml_{i\in\I\e}\|u\|_{\L(Y\ie)}^2+\suml_{i\in\I\e}\|u\|_{\L(R\ie)}^2+
\|u\|^2_{\L(\Omega\setminus\overline{\cupl_{i\in\I\e}\square\ie})}
\\\notag
&\leq
C\max\left\{P\e^{-1};\, Q\e^{-1}\right\}
\a\e[u,u]+ C\eps^2\|\nabla (\wt\J\e^1 u)\|_{\L(\Omega)}^2\\\notag
&
\leq
C\max\left\{P\e^{-1};\, Q\e^{-1};\, \eps^2\right\}\|u\|^2_{\HS^1\e}
\\\notag
&=
C\max\left\{P\e^{-1};\, Q\e^{-1};\, \eps^2\right\}(\A\e u+u,u)_{\L(\Omega\e)}
\\\notag
&\le
C\max\left\{P\e^{-1};\, Q\e^{-1};\, \eps^2\right\}\|\A\e u+u\|_{\L(\Omega\e)}\|u\|_{\L(\Omega\e)}
\\
&=
C\max\left\{P\e^{-1};\, Q\e^{-1};\, \eps^2\right\}\|v\|_{\L(\Omega\e)}\|(\A\e+\Id)^{-1}v\|_{\L(\Omega\e)}
\label{long:formula}
\end{align} 
(note that the constant $C$ in \eqref{long:formula} changes from line to line).
The required estimate \eqref{th5:est} follows immediately from \eqref{long:formula}.
Theorem~\ref{th5} is proven.

{ 

\section{Case \eqref{spec.case} revisited}\label{sec:7}
 
In this section we obtain another
$(\H^1\to \L)$ operator estimate for $P,Q$ satisfying \eqref{spec.case}. 
In some cases (see the discussion after Theorem~\ref{th2a}) the new estimate gives better convergence rate than the estimates \eqref{th1:est} and \eqref{th2:est}.
 
We  define for $P<\infty$:
\begin{gather}\label{eta:prime:2}
\eta\e''\ceq \max\left\{P\e Q\e^{-1/2};\,|P\e-P|;\,\eps;\,\Lambda\e^{n/2}\right\}.
\end{gather}
It is easy to see that  $\eta\e''\to 0$  as $\eps\to 0$
provided \eqref{spec.case} holds.
 
\begin{theorem}\label{th2a}
Let $P,Q$ satisfy \eqref{spec.case}.
Then one has 
\begin{gather}\label{th2a:est} 
\left\|(\A\e+\Id)^{-1}\J\e  - \J\e(\A+\Id)^{-1}\right\|_{\L(\Omega)\to \H^1(\Omega\e)}\le 
C\eta\e''.
\end{gather}
\end{theorem}

Before to present the proof of Theorem~\ref{th2a},
we compare it with the estimates \eqref{th1:est} and \eqref{th2:est}.
Let $P,Q$ satisfy \eqref{spec.case}.
In this case $P\e=P=V$, whence
\begin{gather}\label{kappa}
V\e-V=P\e-P-\kappa\e,\quad\text{where }\kappa\e\ceq{P^2\e\over  Q\e }{1\over 1+P\e Q\e^{-1}}\overset{\text{via }\eqref{spec.case}}\sim {P^2\e\over  Q\e}.
\end{gather}
Using \eqref{kappa}, we get
\begin{align}\label{eta:altern}
\eta\e&=
\max\left\{|P\e-P|+\mathcal{O}(P\e^2 Q\e^{-1});\,\eps;\,R\e\Lambda\e^{n/2} \right\},\quad
R\e\ceq 
\begin{cases}
\Lambda\e^{1-n/2},& n\ge 5,\\[1mm]
\Lambda\e^{-1}|\ln\deps|,&n=4,\\[1mm]
\Lambda\e^{-1},& n= 3,  \\[1mm]
\Lambda\e^{-1}|\ln\deps|^{-{1/2}} ,& n=2,
\end{cases}
\\\label{eta:1:altern}
\eta\e'&= \max\left\{r\e P\e Q\e^{-1/2};\,|P\e-P|+\mathcal{O}(P\e^2 Q\e^{-1});\,\eps;\,R\e\Lambda\e^{n/2}\right\},
\quad r\e\ceq (1+P\e Q\e^{-1})^{-1}.
\end{align}
Taking into account that $R\e\to \infty$ and $r\e\to 1$ as $\eps\to 0$, we 
conclude from \eqref{eta:altern}--\eqref{eta:1:altern} that 
the estimate \eqref{th2a:est} gives a better convergence rate than the estimates \eqref{th1:est} and \eqref{th2:est}
if the term $\Lambda\e^{n/2}$ prevails in $\eta\e''$,  i.e., if  
$P\e Q\e^{-1/2}=o(\Lambda\e^{n/2})$, $|P\e-P|=o(\Lambda\e^{n/2})$, and $\eps=o(\Lambda\e^{n/2})$.
 
We illustrate the above conclusion with an example. Let $n\ge 5$, $\gamma\e=\eps^{n-(n-1)s}$ and $d\e=\eps^s$ with $s\in (1,\frac{n}{n-2})$.  
We have $P\e=\varkappa_n>0$ and $Q\e\to \infty$ which corresponds to the dashed bold open interval on Figure~\ref{fig2}.
One has:
\begin{gather*}
\eta\e=\max\left\{ \mu\e\eps^{n-s(n-2)} ;\,\eps;\ \eps^{s-1}\right\},\
\eta\e'=\max\left\{\nu\e\eps^{n-s(n-2)\over 2}; \mu\e\eps^{n-s(n-2)};\,\eps;\ \eps^{s-1}\right\},\\ 
\eta\e''=\max\left\{\nu\eps^{n-s(n-2)\over 2};\, \eps;\ \eps^{(s-1)n\over 2}\right\},
\end{gather*}
where  $\nu\ceq\sqrt{\varkappa_n\over n-2}$, $\nu\e\sim\nu$, $\mu\e\sim \nu^2$.
Evidently, $\eta\e''=o(\eta\e)$ and $\eta\e''=o(\eta\e')$ for $s\in (1,{n+2\over n})$.

\begin{proof}[Proof of Theorem~\ref{th2a}]

As before, let $\HS\e\ceq \L(\Omega\e)$ and $\HS\ceq \L(\Omega)$, and
the forms $\a\e$ and $\a$ in $\HS\e$ and $\HS$ be defined by \eqref{ae} and \eqref{a}, respectively. 
Again we introduce the spaces $\HS\e^1$, $\HS^1$ as in \eqref{H1}, and the space $\HS^2$ as in \eqref{H2}.
Finally, let the operators $\J\e$, $\wt\J\e$ and $\wt\J\e^1$ be defined by \eqref{J}, \eqref{wtJ} and \eqref{wtJ1prop}, respectively, while this time the operator $\J\e^1$ be defined by \eqref{J1:new}. 
Below we demonstrate that the conditions of Theorem~\ref{thA1} hold with $\delta\e = C\eta\e''$.

Recall that
  \eqref{thA1:0} is fulfilled for any $\delta\e\ge 0$. Moreover (see \eqref{cond2:final}), the condition
 \eqref{thA1:2} is fulfilled with $\delta\e=C\delta_{\eps,1}$, where $\delta_{\eps,1}$
is given in \eqref{delta:1}. Also, since now $\J\e^1 f = \J\e f$ for $f\in\dom(\a)$, we conclude that the condition
\eqref{thA1:1} holds with any $\delta\e\ge 0$. 

We examine the fulfilment of the condition \eqref{thA1:3}.
Let $f\in\HS^2$ and $u\in\HS^1\e$. Recall that 
$f\ie$ and $u\ie$ stand for the mean values of $f$ and $\wt\J\e^1 u$ over 
$\square\ie$. We also denote by $\wt f\ie$ and $\wt u\ie$ the mean 
values of $f$ and $\wt\J\e^1 u$ over $\partial D\ie$, i.e.
$$
\wt f\ie\ceq {1\over \varkappa_n d\e^{n-1}}\int_{\partial D\ie}f\d s,\quad
\wt u\ie\ceq {1\over \varkappa_n d\e^{n-1}}\int_{\partial D\ie}\wt\J\e^1 u\d s.
$$
Taking into account the definition of $P\e$ and that 
$V=P$ (this follows from \eqref{spec.case}), we have
\begin{gather*}
\a\e[u,\J^1\e f]-\a[\wt\J^{1}\e u,f]
=-\suml_{i\in\I\e}(\nabla \J\e^1 u,\nabla f)_{\L(D\ie)}+
\gamma\e\suml_{i\in\I\e}(u,f)_{\L(\partial D\ie)}-P(\J\e^1 u,f)_{\L(\Omega)}
=
\suml_{k=1}^7 L_{\eps,k}, 
\end{gather*}
where $\ds L_{\eps,1}=-\suml_{i\in\I\e}(\nabla \J\e^1 u,\nabla f)_{\L(D\ie)}$, and
\begin{gather*}
\begin{array}{lll}
\ds L_{\eps,2}=({ P\e}-{ P})\suml_{i\in\I\e}u\ie\overline{f\ie}\eps^n,
&
\ds L_{\eps,3}=P\e\suml_{i\in\I\e} ({ \wt u\ie} - { u\ie})\overline{\wt f\ie }\eps^n,\\
\ds L_{\eps,4}=P\e\suml_{i\in\I\e}u\ie\overline{({ \wt f\ie} - { f\ie})}\eps^n,&
\ds
L_{\eps,5}=\gamma\e\suml_{i\in\I\e} (u, f-{ \wt f\ie})_{\L(\partial D\ie)},\\
\ds L_{\eps,6}=P\suml_{i\in\I\e} (\J\e^1 u,{ f\ie}- f)_{\L(\square\ie)},
&
\ds L_{\eps,7}=-P\suml_{i\in\I\e}(u,f)_{\L(\Omega\setminus\cup_{i\in\I\e}\square\ie)}.
\end{array}
\end{gather*}
We have already examined the terms $L_{\eps,1}$ and $L_{\eps,7}$ in \eqref{I1:final} and \eqref{I33:final}, respectively:
\begin{gather}\label{L1}
|L_{\eps,1}|\leq C\delta_{\eps,1}\|f\|_{\HS^2}\|u\|_{\HS\e^1},\\
\label{L7}
|L_{\eps,7}|\leq C\eps^2\|f\|_{\HS^1}\|u\|_{\HS\e^1}.
\end{gather}
Using the Cauchy-Schwarz inequality \eqref{CaSw}   and a similar inequality for $u\ie$, one has
\begin{align}\notag
|L_{\eps,2}|&\leq |P\e-P|\left(\suml_{i\in\I\e}\eps^{n}|f\ie|^2\right)^{1/2}
\left(\suml_{i\in\I\e}\eps^{n}|u\ie|^2\right)^{1/2}\\
&\leq |P\e-P|\|f\|_{\L(\Omega)}\|\wt J\e^1 u\|_{\L(\Omega)}
\leq C|P\e-P|\|f\|_{\HS}\|u\|_{\HS^1\e}\label{L2}
\end{align}
(the last step in \eqref{L2} relies on \eqref{wtJ1prop+}). 
Similarly, using the Cauchy-Schwarz inequality for $\wt f\ie$ and
taking into account the definition of $P\e$, we obtain
\begin{align*} \notag
|L_{\eps,3}|\leq P\e 
\left( P\e^{-1}\ga\e\suml_{i\in\I\e}\|f\|_{\L(\partial D\ie)}^2\right)^{1/2}
\left(\suml_{i\in\I\e}\eps^{n}|\wt u\ie - u\ie|^2\right)^{1/2}.
\end{align*}
Note that, due to \eqref{pq+}, $\sup_{\eps\in (0,\eps_0]}Q\e^{-1}<\infty$ if $Q\not =0$ (cf.~\eqref{spec.case}).
Then, using the estimate
\eqref{lemma:crucial:est+} (for $ f$) and the estimate 
\eqref{lemma:tildeg:est1} (for $ \wt\J\e^1 u$),    we obtain
\begin{align}
|L_{\eps,3}|&\leq 
C\, P\e\left(\suml_{i\in\I\e}\left(\|f\|_{\L(Y\ie)}^2+ Q\e^{-1}\|\nabla f\|_{\L(Y\ie)}^2\right)\right)^{1/2}
\left(\left(\eps^2+Q\e^{-1}\right)\suml_{i\in\I\e}\|\nabla  \wt J\e^1 u\|_{\L(\square\ie)}^2\right)^{1/2}
\notag
\\
&\leq C\, P\e (\eps^2+Q\e^{-1})^{1/2} \|f\|_{\H^1(\Omega)}\|\nabla \wt J\e^1 u\|_{\L(\Omega)}
\leq C\, \max\{P\e\eps;\,P\e Q\e^{-1/2}\}  \|f\|_{\HS^1}\|  u\|_{\HS^1\e}\label{L3}
\end{align}
(in the last step we again use  \eqref{wtJ1prop+}). 
Further, using  the estimate \eqref{lemma:tildeg:est1} for $f$, the Cauchy-Schwarz 
inequality for $u\ie$, and \eqref{wtJ1prop+}, we derive
\begin{align} \notag
|L_{\eps,4}|&\leq P\e  
\left(\eps^n\suml_{i\in\I\e}|\wt f\ie- f\ie|^2\right)^{1/2}\left(\suml_{i\in\I\e}\eps^{n}| u\ie|^2\right)^{1/2}
\\
&\leq C\, P\e (\eps^2+Q\e^{-1})^{1/2} \|\nabla f\|_{\L(\Omega)}\| \wt J\e^1 u\|_{\L(\Omega)}
\leq C\, \max\{P\e\eps;\,P\e Q\e^{-1/2}\}\|f\|_{\HS^1}\|  u\|_{\HS^1\e}.\label{L4}
\end{align}
To estimate $L_{\eps,5}$ we utilize  \eqref{wtJ1prop+},
\eqref{lemma:tildeg:est2}  
and
\eqref{lemma:crucial:est+}:
\begin{align}
\notag
|L_{\eps,5}|&\leq 
\left(\gamma\e \suml_{i\in\I\e}\|f- \wt f\ie\|_{\L(\partial D\ie)}^2\right)^{1/2}
\left(\gamma\e\suml_{i\in\I\e}      \|u\|_{\L(\partial D\ie)}^2\right)^{1/2}
\\\notag
&\leq
 C\left(\gamma\e d\e\suml_{i\in\I\e}\|\nabla f\|^2_{\L(D\ie)} \right)^{1/2} 
 \left(P\e\suml_{i\in\I\e}\left(\|u\|_{\L(Y\ie)}+Q\e^{-1}\|\nabla u\|_{\L(Y\ie)}\right)\right)^{1/2}
\\
&\leq
 C\,P\e Q\e^{-1/2}\|f\|_{\HS^1}\|u\|_{\HS\e^1} .
\label{L5}
\end{align}
The last step of the above estimate relies on 
 $\sup_{\eps\in (0,\eps_0]}Q\e^{-1}<\infty$ (for $Q$ as in \eqref{spec.case}), and
the equalities
$$
\gamma\e d\e = P\e Q\e^{-1}(n-2),\ n\ge 3,
\qquad
\gamma\e d\e = P\e Q\e^{-1}|\ln d\e|^{-1},\ n=2
$$
(note   $|\ln d\e|^{-1}< |\ln\eps_0|^{-1}$).
Finally, we estimate $L_{\eps,6}$ via   \eqref{Poincare}
(with $g\ceq f-f\ie$) and \eqref{wtJ1prop+}:
\begin{align}
\notag
|L_{\eps,6}|&\leq 
P\left(\suml_{i\in\I\e}\|f- f\ie\|_{\L(\square\ie)}^2\right)^{1/2}
\left(\suml_{i\in\I\e} \|\wt\J\e u\|_{\L(\square\ie)}^2\right)^{1/2}
\\&
\leq C\eps\|\nabla f\|_{\L(\Omega)}\|\wt \J\e u\|_{\L(\Omega)}\leq
C\eps\|f\|_{\HS^1}\|u\|_{\HS^1\e}.
\label{L6}
\end{align}
\smallskip
Combining \eqref{L1}--\eqref{L6} and taking into account that $\sup_{\eps\in(0,\eps_0)}P\e<\infty$, we arrive at
\begin{gather*}\label{L:all}
|\a\e[u,\J^1\e]-\a[\wt\J^{1}\e u,f]|\leq 
C\max\left\{P\e Q\e^{-1/2};\, |P\e-P|;\ \eps;\ \delta_{\eps,1} \right\}\|f\|_{\HS^2}\|u\|_{\HS\e^1},
\end{gather*}

Taking into account that $\delta_{\eps,1}=\max\{\Lambda\e^{n/2};\, (\eps d\e)^{1/2}\}\leq 
\max\{\Lambda\e^{n/2};\, \eps\}$, we conclude  that
the
conditions \eqref{thA1:0}--\eqref{thA1:3}  are fulfilled with
$ \delta\e=C \eta\e''$, where $\eta\e''$ is defined by \eqref{eta:prime:2}.
Hence, by Theorem~\ref{thA2}, we obtain the estimate
\begin{gather}\label{th2a:est:prelim} 
\left\|(\A\e+\Id)^{-1}\J\e  - \J\e^1(\A+\Id)^{-1}\right\|_{\HS \to \HS^1\e}\le 
C\eta\e''.
\end{gather} 
Since $\|v\|_{\H^1(\Omega\e)}\leq \|v\|_{\HS^1\e}$ and $\J\e^1 f = \J\e f $, 
we immediately get from \eqref{th2a:est:prelim} the required estimate 
\eqref{th2a:est}. Theorem~\ref{thA1+} is proven.
\end{proof}

}

\section*{Acknowledgements}

The work of the first author is partly supported by the Czech Science Foundation (GA\v{C}R) through the project 21-07129S.  The second author gratefully acknowledges financial support by 
Deutsche Forschungsgemeinschaft (DFG, German Research Foundation) -- Project-ID 258734477 -- SFB 1173.

\end{document}